\documentclass[12pt]{article}
\usepackage{aurical}
\usepackage[T1]{fontenc}
\usepackage{amssymb,amsmath,amsthm,bm}
\usepackage{setspace,mathrsfs,nicefrac}
\usepackage{bbm,url,color,wasysym}
\usepackage[margin=1in]{geometry}
\usepackage{paralist}\setdefaultitem{$-$}{\tiny $\bullet$}{\tiny$\circ$}{$\star}

\newcommand{\IN}{\mathbb{N}}

\newcommand{\R}{\mathbb{R}}
\newcommand{\C}{\mathbb{C}}
\newcommand{\cC}{\mathcal{C}}
\newcommand{\cI}{\mathcal{I}}
\newcommand{\cG}{\mathcal{G}}
\newcommand{\cN}{\,\text{\Fontlukas{N}}}
\newcommand{\cB}{\mathcal{B}}
\newcommand{\cT}{\mathcal{T}}
\newcommand{\cM}{{\text{\Fontlukas{M}\,}}}

\newcommand{\Lip}{\text{\rm L}}
\renewcommand{\P}{\mathrm{P}}
\newcommand{\E}{\mathrm{E}}
\newcommand{\F}{\mathcal{F}}
\renewcommand{\L}{\mathbb{L}}

\newcommand{\1}{\mathbbm{1}}
\renewcommand{\d}{{\rm d}}

\newcommand{\e}{{\rm e}}
\renewcommand{\geq}{\geqslant}
\renewcommand{\leq}{\leqslant}
\renewcommand{\ge}{\geqslant}
\renewcommand{\le}{\leqslant}

\author{Robert C. Dalang\thanks{The research of R.C.D.~is supported in part by the Swiss National Foundation for Scientific Research.}\\EPF--Lausanne\and Davar Khoshnevisan\thanks{%
	Research of D.K. was supported in part by  the United States
	National Science Foundation grants DMS-1006903,
	DMS-1307470,  DMS-1608575, and PHY11-25915
	through an NSF grant to the Kavli Institute for Theoretical Physics
	at UC Santa Barbara. Parts of this research were carried out while the
	authors were at the Banff International Research Station in April 2012.}
	\\University of Utah\and Tusheng Zhang\\University of Manchester}
\title{\bf\Large{Global solutions to stochastic reaction-diffusion equations
	with super-linear drift and multiplicative noise}}
\date{}
\newtheorem{stat}{Statement}[section]
\newtheorem{proposition}[stat]{Proposition}

\newtheorem{theorem}[stat]{Theorem}

\newtheorem{lemma}[stat]{Lemma}
\theoremstyle{definition}
\newtheorem{definition}[stat]{Definition}
\newtheorem{remark}[stat]{Remark}

\numberwithin{equation}{section}
\begin{document}
\maketitle
\begin{abstract}
	Let $\xi(t\,,x)$ denote space-time white noise
	and consider a reaction-diffusion equation of the form
	\[
		\dot{u}(t\,,x)=\tfrac12 u''(t\,,x) + b(u(t\,,x)) + \sigma(u(t\,,x)) \xi(t\,,x),
	\]
	on $\R_+\times[0\,,1]$,
	with homogeneous Dirichlet boundary conditions and suitable initial data, 
	in the case that there exists $\varepsilon>0$
	such that $\vert b(z)\vert \ge|z|(\log|z|)^{1+\varepsilon}$
	for all sufficiently-large values of $|z|$. 
	When $\sigma\equiv 0$, it is well known that such PDEs
	frequently have non-trivial stationary solutions. By contrast,
	Bonder and Groisman \cite{BonderGroisman} have recently shown that
	there is finite-time blowup when 
	$\sigma$ is a non-zero constant. In this paper,
	we prove that the Bonder--Groisman condition is unimproveable
	by showing that the reaction-diffusion equation with noise is ``typically''
	well posed when $\vert b(z) \vert =O(|z|\log_+|z|)$ as $|z|\to\infty$.
	We interpret the word ``typically''
	in two essentially-different ways without altering the conclusions of
	our assertions.
\end{abstract}
\vskip.5in 



\noindent{\it Abbreviated title:} SPDEs with super-linear drift

\section{Introduction}

Let $\xi$ denote space-time white noise on
$\R_+\times[0\,,1]$, and consider the
parabolic stochastic partial differential equation
\begin{equation}\label{SHE}
	\dot{u}(t\,,x) = \tfrac12 u''(t\,,x)
	+ b(u(t\,,x)) + \sigma(u(t\,,x)) \xi(t\,,x),
\end{equation}
$t>0$, $x \in (0\,,1)$, subject to the homogeneous Dirichlet boundary condition,
\[
	u(t\,,0)=u(t\,,1)=0\qquad\text{for all $t>0$},
\]
and the initial condition $u(0\,,\cdot) = u_0$ on $[0\,,1]$.
Throughout, $\sigma$, $b$, and $u_0$ are assumed to be nonrandom and
measurable real-valued functions on the real line.

It is  well known that if, in addition, $b$, $\sigma$ are globally Lipschitz
functions then any local solution of \eqref{SHE} is necessarily
a global one.  Note that the Lipshitz continuity of $\sigma$ and $b$
implies their sublinear growth; that is,
$|b(z)|+ |\sigma(z)|=O(|z|)$ as $|z|\to\infty$.
In 2009, Bonder and Groisman \cite{BonderGroisman}
proved the following interesting complement.

\begin{theorem}[Bonder and Groisman \protect{\cite{BonderGroisman}}]\label{th:BG}
	Suppose, in addition, that $\sigma$ is a nonzero constant,
	$b\ge0$ is convex 
	and satisfies 
	$\int_1^\infty \d z/b(z)<\infty$, 
	and the initial function $u_0$ is nonnegative,
	continuous on $[0\,,1]$, and vanishes on
	$\{0\,,1\}$. Then there exists a random time $\tau$ such that
	$\P\{\tau<\infty\}=1$ and
	\[
		\lim_{t \uparrow \tau}\int_0^1|u(t\,,x)|^2\,\d x= \infty
		\qquad\text{almost surely}.
	\]
\end{theorem}

\begin{remark}
	To be precise, Theorem 3.2 of Bonder and Groisman
	\cite{BonderGroisman} implies the weaker conclusion that
	$\lim_{t \uparrow \tau}\sup_{x\in[0,1]}|u(t\,,x)|=\infty$ a.s.
	However, their proof yields the stronger result that
	\[
		\Phi(t) := \int_0^1 u(t\,,x)\sin(\pi x)\,\d x
	\]
	explodes in
	finite time a.s.; see the discussion prior to the statement of Lemma 3.1
	in \cite{BonderGroisman}. This and the Cauchy--Schwarz inequality
	together yield Theorem \ref{th:BG}.
\end{remark}

Theorem \ref{th:BG} is surprising because, if we set $\sigma\equiv0$, then the resulting reaction--diffusion
equation \eqref{SHE} can have non-trivial
global stationary solutions \cite{BonderGroisman,GV,SGKM}.
Therefore, we see that the introduction of any amount of additive
space-time white noise to a reaction diffusion
equation removes the possibility of global well-posedness if the reaction term grows
faster than a constant multiple of $|z|(\log |z|)^{1+\varepsilon}$
for some $\varepsilon>0$ (say) as either
$z\to\infty$ or $z\to-\infty$.
	
	Several papers in the literature discuss stochastic pde's with locally Lipschitz coefficients that have polynomial growth and/or satisfy certain monotonicity conditions (see \cite{cerrai-1,DMP,LR}, for instance). The typical example of such a coefficient  is $b(u) = -u^3$, which has the effect of ``pulling the solution back toward the origin.'' This is quite different from the situation that we discuss in this paper, where $b(u)$ will typically ``push'' the solution towards $\pm\infty$.

The goal of this article is to prove that 
the Bonder--Groisman
theorem (Theorem \ref{th:BG})
is optimal. In fact, we introduce two rather different methods
that show that, under two different sets of natural conditions,
if $|b(z)|=O(|z|\log|z|)$, then \eqref{SHE} is globally well-posed.
With this aim in mind, let us first introduce some notation.

\begin{definition}\label{def_L2}
	Suppose that $u_0 \in L^2[0\,,1]$.  We say that
	\eqref{SHE} has an \emph{$\L^2_{\textit{loc}}$-solution}
	$u$ if there exists a stopping time $\tau$
	[with respect to the standard Brownian filtration
	generated by $\xi$; see {\it Step 2}, following \eqref{1:T}
	below] and an adapted random field $\{u(t\,,\cdot)\}_{t\in [0\,, \tau)}$
	such that, almost surely on the event $\{\tau > t\}$,
	\begin{align*}
		\int_0^1u(t\,,x)\phi(x)\,\d x&=\int_0^1u_0(x)\phi(x)\,\d x
			+\frac12\int_{(0,t)\times(0,1)} u(s\,,x)\phi''(x)\,\d s\,\d x\\
		&\qquad +\int_{(0,t)\times(0,1)} b(u(s\,,x))\phi(x)\,\d s\,\d x
			+\int_{(0,t)\times(0,1)}\sigma(u(s\,,x))\phi(x)\,\xi(\d s\,\d x),
	\end{align*}for every 
	non-random test function 
	$\phi\in C^2[0\,,1]$ that satisfies $\phi(0) = \phi(1) = 0$,
\end{definition}

Our first result can be stated as follows.

\begin{theorem}\label{th:L2}
	Suppose in addition that $u_0\in L^2[0\,,1]$, $\sigma:\R\to\R$
	is bounded,
	and $|b(z)|=O(|z|\log|z|)$ as $|z|\to\infty$.
	Then, every $\L^2_{\textit{loc}}$-solution $u$ of \eqref{SHE}
	is a long-time solution; that is,
	\[
		\sup_{t\in[0,\tau\wedge T]}\int_0^1|u(t\,,x)|^2\,\d x<\infty \qquad
		\text{a.s.\ for every $T\in(0\,,\infty)$}.
	\]
\end{theorem}

\begin{remark} 
	Suppose $\tau$ is a maximal time up to which the solution can be constructed;
	that is,
	\[
	   \sup_{t \in [0,\tau)} \Vert u(t) \Vert_{L^2[0,1]} = \infty\qquad \text{a.s.}
	\]
	Then Theorem \ref{th:L2} implies
	that  $\tau = \infty$ a.s. In this case, 
	$\sup_{t \in [0,T]} \Vert u(t) \Vert_{L^2[0,1]} < \infty$ a.s.\ for all $T >0$. 
	The question of the existence of an $\L^2_{\textit{loc}}$-solution of 
	\eqref{SHE} under the assumptions of Theorem \ref{th:L2} is open.
\end{remark}

Theorem \ref{th:L2} is an infinite-dimensional variation on
aspects of the theory of Fang and Zhang \cite{FangZhang} on
stochastic differential equations with superlinear coefficients.
We follow the Lyapunov function method of Fang and Zhang,
and overcome
the transition from finite to infinite dimensions by appealing to 
the sharp form of Gross'
logarithmic Sobolev inequality for normalized Lebesgue measure \cite{Gross}. We believe that
our technique might also have other uses in infinite-dimensional stochastic analysis.

Our second result is based on an $L^\infty$ method and, as such,
requires stronger regularity on the initial function $u_0$. In order
to introduce it, we first need some notation.

\begin{definition}
	For every $\alpha\in(0\,,1)$, we define $\C^\alpha_0$ to be the collection of
	all functions $f:[0\,,1]\to\R$ such that $f(0)=f(1)=0$ and
	\[
		\|f\|_{\C^\alpha_0}:=
		\sup_{0\le x<y\le 1}\frac{|f(y)-f(x)|}{|y-x|^\alpha} < \infty.
	\]
	The space $\C^1_0$ will denote the collection of
	all Lipschitz-continuous functions $f:[0\,,1]\to\R$ such that $f(0)=f(1)=0$.
	We sometimes write Lip$(f)$ in place of $\|f\|_{\C^1_0}$.
\end{definition}

In all cases, we see that $\C^\alpha_0$ is a Banach space endowed with the
norm $\|\cdot\|_{\C^\alpha_0}$.
Let us also recall the following definition (see Dalang \cite{Dalang1999}).

\begin{definition}\label{def1.7}
	A {\em random field solution} to \eqref{SHE} is a jointly measurable
	and adapted space-time process $u:=\{u(t\,,x)\}_{(t,x) \in \R_+ \times [0,1]}$
	such that, for all $(t\,,x) \in \R_+ \times [0\,,1]$,
	\begin{align*}
		u(t\,,x) & =  (\cG_tu_0)(x) + \int_{(0,t)\times (0,1)}
			G_{t-s}(x\,,y)  b(u(s\,,y))\, \d s\, \d y \\
		& \hskip 2in + \int_{(0,t)\times (0,1)} G_{t-s}(x\,,y)
			\sigma(u(s\,,y))\, \xi (\d s\, \d y),
	\end{align*}
	a.s.,
	where $\{\cG_t\}_{t\ge0}$ and $G$ are respectively the heat semigroup
	and heat kernel for the Dirichlet Laplacian, and are recalled in \eqref{G_t}
	and \eqref{G} below.
\end{definition}

\begin{remark} \label{rd_rem1.8}
	The stochastic integral in Definition \ref{def1.7} is not 
	always defined in the sense of Walsh \cite{Walsh}
	since the Walsh integral is defined provided  that
	for all $t>0$ and $x\in[0\,,1]$,
	\[
		\E\left[ \int_{(0,t)\times (0,1)} \left[G_{t-s}(x\,,y)
		\sigma(u(s\,,y))\right]^2\, \d s\, \d y \right] < \infty.
	\]
	Instead, we are using a localized version of the Walsh integral,
	for whose existence we  require only that
	for all $t>0$ and $x\in[0\,,1]$,
	\[
		\int_{(0,t)\times (0,1)} \left[G_{t-s}(x\,,y)
		\sigma(u(s\,,y))\right]^2\, \d s\, \d y  < \infty \qquad \text{a.s.}
	\]
\end{remark}



We are now in position to present our second complement to the Bonder--Groisman
theorem [Theorem \ref{th:BG}].

\begin{theorem}\label{th:C:alpha}
	Suppose that:
	\begin{compactenum}[i.]
	\item $u_0\in\cup_{0<\alpha\le1}\C^\alpha_0$;
	\item $b$ and $\sigma$ are locally Lipschitz functions such that
		\begin{equation}\label{rd1.2a}
			|b(z)|=O(|z|\log|z|)
			\quad\text{and}\quad
			\left\vert\sigma(z)\right\vert = o\left(|z| (\log |z| )^{1/4}\right),
		\end{equation}
		as $|z|\to\infty$.
	\end{compactenum}
	Then the SPDE \eqref{SHE} has a random field solution $u$ in
	$C(\R_+ \times [0,1])$, and this solution is unique. In particular, $u$ satisfies
	\begin{equation}\label{rd0.1}
		\sup_{t\in[0,T]}\sup_{x\in[0,1]}|u(t\,,x)|<\infty\qquad
		\text{ a.s.\ for all $T\in(0\,,\infty)$}.
	\end{equation}
\end{theorem}

   Recall that \eqref{rd1.2a} means that there is a constant $C < \infty$ such that for $|z|$ large enough, $\vert b(z)\vert \leq C |z|\log|z|$, and $\lim_{|z| \to \infty } \vert \sigma(z)/(z (\log |z| )^{1/4}) \vert = 0$.
	
We conclude the introduction with a few words about the notation
that is used consistently in this paper.\\

\noindent\textbf{Remarks on notation.}
\begin{compactenum}
\item Throughout the paper, we write $\L^p$ in place of $L^p[0\,,1]$
	for every $1\le p\le\infty$. In particular, $\|f\|_{\L^\infty}$
	and $\|f\|_{\L^2}$ respectively denote the essential supremum
	and the $\L^2$-norm of a suitable function $f:[0\,,1]\to\R$.
\item Throughout, we define
	$\log_+(w) := \log(\max(w\,,\e))$ for all $w\in\R$.
\item If $f$ and $g$ are non-negative functions on some space $X$,
	then we write $f(x)\apprle g(x)$ for
	all $x\in X$  [equivalently, $g(x)\apprge f(x)$ for all $x\in X$]
	to mean that there exists a finite constant $A$ such that
	$f(x)\le Ag(x)$ for all $x\in X$.
\item If $f$ and $g$ are non-negative functions on some space $X$,
	then  we write $f(x)\asymp g(x)$ for all $x\in X$
	to mean that $f(x)\apprle g(x)\apprle f(x)$ for all $x\in X$.
\end{compactenum}

\section{Proof of Theorem \ref{th:L2}}

We will appeal to the logarithmic Sobolev inequality of
Gross \cite{Gross} in the following form:
For every $\varepsilon\in(0\,,1)$ and differentiable
functions $h:[0\,,1]\to\R$ that vanish continuously on $\{0\,,1\}$,
\[
	\int_0^1 |h(x)|^2\log|h(x)|\,\d x \le \varepsilon\|h'\|_{\L^2}^2
	+\tfrac14\log(1/\varepsilon) \|h\|_{\L^2}^2
	+\|h\|_{\L^2}^2\log\left(\|h\|_{\L^2}^2\right),
\]
where $0\log 0 :=0$.
One can derive this logarithmic Sobolev inequality  
from formula (5.4) in ref.\ \cite{Gross} using the fact that
\[
	\log\left( \|\cG_{\varepsilon}\|_{2\rightarrow\infty}\right)\leq
	\tfrac12\sup_{x\in[0,1]}\log\left(\int_0^1 \vert G_{\varepsilon}(x\,,y)\vert^2\,\d y
	\right) \leq \tfrac14\log(1/\varepsilon),
\]
where $\cG:=\{\cG_t\}_{t\ge0}$ denotes the
heat semigroup and $G:(0\,,\infty)\times[0\,,1]^2\to\R_+$
the corresponding heat kernel; see \eqref{G} below.

Let $\L^2\log\L$ denote the vector space of all measurable functions
$h:[0\,,1]\to\R$ that satisfy
\[
	\|h\|_{\L^2\log\L} :=\left(\int_0^1
	|h(x)|^2\log_+|h(x)|\,\d x\right)^{1/2}<\infty.
\]
Now, $\int_0^1 |h(x)|^2\1_{\{|h(x)|<\e\}}\,\d x	\le \|h\|_{\L^2}^2$
and
\[
	\int_0^1|h(x)|^2|\log|h(x)|| \1_{\{0< |h(x)|\le 1\}}\,\d x
	\le \int_0^1 |h(x)|\log(1/|h(x)|) \1_{\{0< |h(x)|\le 1\}}\,\d x
	\le \e^{-1},
\]
since $y\log(1/y)\le \e^{-1}$ for all $y\in[0\,,1]$. Therefore,
\begin{align*}
	\|h\|_{\L^2\log\L} &= 
	  	\int_0^1 |h(x)|^2\, \1_{\{|h(x)|\le\e\}}\, \d x + 
		\int_0^1	 |h(x)|^2\log |h(x)|\,\1_{\{|h(x)|>\e\}}\,\d x \\
	& \le \|h\|_{\L^2} + \int_0^1	|h(x)|^2\log |h(x)| \,\d x \\
	&\qquad  - \int_0^1 |h(x)|^2\log |h(x)|\,\1_{\{0<|h(x)|<1\}}\,\d x 
		- \int_0^1	 |h(x)|^2\log |h(x)|\,\1_{\{1\le |h(x)|\le\e\}}\,\d x\\
	& \le  \|h\|_{\L^2} +\int_0^1	|h(x)|^2\log |h(x)| \,\d x + \e^{-1}.
\end{align*}
Together with these
remarks, and a standard density argument,
Gross' logarithmic Sobolev inequality can be cast in the following
manner in terms of $\L^2\log\L$ norms.

\begin{theorem}[The logarithmic Sobolev inequality]\label{th:Gross}
	If $h\in W^{1,2}[0\,,1]$  vanishes continuously
	at the boundary, then
	\[
		\|h\|_{\L^2\log\L}^2 \le \varepsilon\|h'\|_{\L^2}^2
		+K_\varepsilon \|h\|_{\L^2}^2
		+ \|h\|_{\L^2}^2\log_+\left(\|h\|_{\L^2}^2\right)+\e^{-1},
	\]
	for every $\varepsilon\in(0\,,1)$,
	where $K_\varepsilon := 1 + \tfrac14\log(1/\varepsilon)$.
\end{theorem}

We are ready to verify Theorem \ref{th:L2}.

\begin{proof}[Proof of Theorem \ref{th:L2}]
	For every $R>0$, consider the stopping times
	\[
		\tau(R) := \begin{cases}
		          \inf\left\{ t \in [0\,,\tau):\, 
		          \|u(t)\|_{\L^2}> R\right\} & 
		          \text{ if } \{\,\cdots \} \neq \varnothing, \\
		          \tau & \text { otherwise}.
		\end{cases}
	\]
	We aim to prove that 
	$\P\{\sup_{t < \tau \wedge T}  \Vert u(t) \Vert_{\L^2} = \infty\} = 0$. 
	
	Since
	\[
	   \left\{\sup_{t < \tau \wedge T}  \Vert u(t) \Vert_{\L^2} = \infty\right\} 
	   = \bigcap_{R > 0} \{\tau(R) < \tau \wedge T \},
	\]
	it suffices to prove that
	\begin{equation}\label{e_suff}
	   \lim_{R \to \infty} \P\{\tau(R) < \tau\wedge T \} = 0 \qquad \text{ for all } T>0.
	\end{equation}
	
	For every constant $R>0$,
	consider the following stochastic PDE with random forcing and no reaction term:
	\begin{equation}\label{v:L2}
		\dot{v}_R(t\,,x) = \tfrac12 v_R''(t\,,x) + \sigma(u(t\wedge\tau(R)\,,x))\xi(t\,,x)
		\qquad[t >0,\, 0\le x\le 1].
	\end{equation}
	We consider \eqref{v:L2}
	subject to the initial condition $v_R(0)=0$ and the same boundary 
	conditions as \eqref{SHE}; that is, $v_R(t\,,0)=v_R(t\,,1)=0$ for all
	$t>0$. The solution process $t\mapsto v_R(t)$ exists, is unique
	[in $\L^2$], and is an $\L^2$-valued
	stochastic process that satisfies the following weak random integral
	equation viewed as an equation in $\L^2$:
	\begin{equation}\label{eq:v}
		v_R(t\,,x) = \int_{(0,t)\times(0,1)} G_{t-s}(x\,,y)
		\sigma(u(s\wedge\tau(R)\,,y))\,\xi(\d s\,\d y).
	\end{equation}
	Notice that the stochastic integral is well-defined as a 
	Walsh integral because $\sigma$ is assumed to be bounded.
	
	In the above, the function $G:(0\,,\infty)\times[0\,,1]^2\to\R_+$ denotes
	the heat kernel; as was mentioned earlier, we will use $\cG:=\{\cG_t\}_{t\ge0}$ 
	to denote the
	corresponding heat semigroup. That is, $\cG_0 f=f$, for $t>0$,
	\begin{equation}\label{G_t}
		(\cG_tf)(x) := \int_0^1 G_t(x\,,y) f(y)\,\d y,
	\end{equation}
	and $(t\,,x\,,y)\mapsto G_t(x\,,y)$ denotes the
	fundamental solution to the heat equation
	$\dot{G} = \frac12 G''$ on $(0\,,\infty)\times[0\,,1]$ with
	zero boundary conditions, viz.,
	\begin{equation}\label{G}
		G_t(x\,,y) = 2\sum_{n=1}^\infty\sin(n\pi x)\sin(n\pi y)\exp\left(-\frac{
		n^2\pi^2 t}{2}\right),
	\end{equation}
	for every $t>0$ and $x,y\in[0\,,1]$. We recall the
	well-known inequality, valid for all $t>0$ and $x,y \in [0,1]$:
	\begin{equation}\label{boundG}
	   0 \le G_t(x\,,y) \leq p(t\,,x-y),
	\end{equation}
	where $p(t\,,\cdot)$ denotes the standard $N(0\,,t)$ probability density function.
	The preceding assertion is an immediate consequence of
	the classical fact that the mapping
	$(0\,,\infty)\times[0\,,1]^2\ni (s\,,a\,,b)\mapsto G_s(a\,,b)$
	describes the transition densities of a Brownian motion
	killed upon reaching $\{0\,,1\}$; see Bass \cite[Ch.\ 2, \S 7]{Bass}.

	Define, for every fixed $R>0$,
	\[
		d_R := u - v_R.
	\]
	We may observe that $d_R$ is an $\L^2_{\textit{loc}}$-solution of the following
	heat equation: On $\{\tau > t\}$,
	\begin{equation}\label{d}
		\dot{d}_R (t) = \tfrac12d_R'' (t) + b\left( v_R(t)+d_R(t)\right),
	\end{equation}
	subject to initial condition $d_R(0)=u_0$ and 
	boundary values $d_R(t\,,0) = d_R(t\,,1) =0$
	for all $t>0$. It should be emphasized that \eqref{d} is an 
	ordinary partial differential equation with a random coefficient.
	
	Choose and fix some $T>0$.
	Since $\sigma$ is a bounded measurable function, standard estimates
	\cite[Chapter 3]{Walsh} show that there exists $\beta>0$
	such that for all $p\in(2\,,\infty)$,
	$x,y\in[0\,,1]$, and $s,t\in[0\,,T]$,
	\begin{equation}\label{Kolm}
		\E\left(|v_R(t\,,x) - v_R(s\,,y)
		|^p\right)\apprle |t-s|^{p\beta} + |x-y|^{p\beta},
	\end{equation}
	where the implied constant depends only on $p$ and $T$. 
	Since $v_R(0)$ vanishes on $\{0\,,1\}$, a
	standard application of the Kolmogorov continuity theorem
	for random fields \cite[p.\ 31]{Kunita} then shows that
	\begin{equation}\label{2.10}
		A_T :=\sup_{R>0}\E\left(\sup_{t\in[0,T]}\sup_{x\in[0,1]}\left|
		v_R(t\,,x)\right|\right) <\infty.
	\end{equation}
	
	Consider the stopping time
	\[
		\tau_M(R) := \inf\left\{ t>0:\
		\sup_{x\in[0,1]}\left| v_R(t\,,x)\right|
		>M\right\}
		\qquad\text{for every $M>0$}.
	\]
	It follows from \eqref{2.10} and the Chebyshev inequality that
	\begin{equation}\label{2.11}
		\sup_{R>0}
		\P\left\{ \tau_M(R) < T\right\} \le
		\frac{A_T}{M}.
	\end{equation}
	
	Next, we observe that \eqref{Kolm} and a suitable form
	of the Kolmogorov continuity theorem \cite[p.\ 31]{Kunita}
	together show also that $v_R(\cdot,\cdot)$ has a version with continuous sample paths a.s.
	The process $u(\cdot\wedge\tau(R))(\cdot)$ also has a jointly continuous version (for $t>0$ and $x\in [0,1]$). Indeed, from the ``weak'' formulation of Definition \ref{def1.7}, one deduces as in \cite[Chapter III]{Walsh} that $u(t)$ is also a mild solution of \eqref{SHE}, which is $L^2$-bounded prior to time $\tau(R)$, so using the growth condition on $b$ and the fact that $\sigma$ is bounded, one easily checks the conditions of the Kolmogorov continuity theorem to deduce the existence of a jointly continuous version (for $t>0$, $x \in [0,1]$) of $u(\cdot\wedge\tau(R))(\cdot)$.

	
	Define two random space-time functions $D$ and $V$ as
	\[
		D(t) := d_R\left( t\wedge\tau(R)\wedge\tau_M(R)\right),
		\quad
		V(t) := v_R\left( t\wedge\tau(R)\wedge\tau_M(R)\right)
		\qquad[0\le t\le T],
	\]
	all the time suppressing the dependence of $D$ and $V$ on $(R\,,M)$,
	as well as on the spatial variable $x\in[0\,,1]$. 
	
	In order to show that for $s>0$, $D(s)$ has some regularity as a function of $x$, let $H_0^1$  denote the completion of $C_0^{\infty}(0\,,1)$ under the Sobolev norm $\|u\|_{H_0^1} := \{\int_0^1\vert u'(x)\vert^2\,\d x\}^{1/2}$.
	We claim that for $t>0$, $D(t) \in H_0^1$, that is, 
	\begin{equation}\label{claim:D':L2}
		D'(t)\in \L^2\qquad\text{a.s., for all $t>0$.}
	\end{equation}
	In order to verify \eqref{claim:D':L2}, recall that 
	\begin{equation}\label{2.12-1}
		\|\cG_tf\|_{H_0^1}\leq t^{-1/2}\|f\|_{\L^2}
		\qquad\text{for every $f\in \L^2$};
	\end{equation}
	see for example Cerrai \cite[(2.6)]{cerrai}.
	Moreover \cite[(2.7)]{cerrai}, 
	\begin{equation}\label{2.12-2}
		\|\cG_tf\|_{\L^2}\leq c_p t^{-(2-p)/(4p)}\|f\|_{\L^p}
		\qquad\text{for every $f\in \L^p$ and $p\in(1\,,2]$}.
	\end{equation}
	The $L\log L$ growth of $b$ implies that for all $p\in(1\,,2)$,
	\begin{equation}\label{2.12-3}
	\|b(V(s)+  D(s))\|_{\L^p}\leq C_{p,M}(1+\|D(s)\|_{\L^2})
	\qquad\text{for all $s>0$},
	\end{equation}
	where $C_{p,M}$ is a finite constant.
	
	   The mild formulation of \eqref{d} yields
	\begin{equation}\label{2.12-4}
		D(t) = \cG_t u_0+ \int_0^t \cG_{t-s}\left(b(V(s)+  D(s))\right)\,\d s
		\qquad[t>0].
	\end{equation}
	Therefore, 
	\begin{equation*}
	   \|D(t)\|_{H_0^1} \leq \|\cG_t u_0\|_{H_0^1} + \int_0^t \left\|  \cG_{t-s}\left(b(V(s)+  D(s))\right)\right\|_{H_0^1}
			\d s.
	\end{equation*}
  By \eqref{2.12-1} and the semigroup property of $t\mapsto\cG_t$, the right-hand side is bounded above by
	\begin{equation*}
	   t^{-1/2}\|u_0\|_{\L^2}
			+ \int_0^t \left(\frac{2}{t-s}\right)^{1/2} 
			\left\|  \cG_{(t-s)/2}\left(b(V(s)+  D(s))\right)\right\|_{\L^2}
			\d s.
	\end{equation*}
	By \eqref{2.12-2}, then \eqref{2.12-3}, for every $p\in(1\,,2)$ and $t>0$, this is bounded above by 
		\begin{align*}
		& t^{-1/2}\|u_0\|_{\L^2}+ c_p\int_0^t 
		    	\left(\frac{2}{t-s}\right)^{1/2} 
			\left(\frac{2}{t-s}\right)^{(2-p)/(4p)}\|b(V(s)+  D(s))\|_{\L^p}\,\d s\\
		&\qquad\leq t^{-1/2}\|u_0\|_{\L^2}+ 
			c_pC_{p,M}\left(1+\sup_{0\leq s\leq t}\|D(s)\|_{\L^2}\right)  
			\cdot\int_0^t\left(\frac{2}{t-s}\right)^{(2+p)/(4p)}\d  s,
	\end{align*}
	which is finite since $(2+p)/(4p) <1$. This  implies \eqref{claim:D':L2}.
	
	Now that we have proved \eqref{claim:D':L2},
	we may combine \eqref{d} with the chain rule of \cite[Lemma 1.1]{pardoux} in order to see that for every
	$t\in[0\,,T]$, 
	\begin{align}\nonumber
		\|D(t) \|_{\L^2}^2 &=  \|u_0\|_{\L^2}^2 
			+ 2 \int_0^t \langle \dot{D}(s)\,,  D(s) \rangle_{\L^2}\, \d s \\ \nonumber
		&=\|u_0\|_{\L^2}^2 + 2 
			\int_0^t \langle \tfrac12 D''(s) \,, D(s) \rangle_{\L^2}\, \d s  
			+ 2 \int_0^t \langle b(V(s)+  D(s))\,,D(s) \rangle_{\L^2}\, \d s\\
		&=\|u_0\|_{\L^2}^2- \int_0^t\left\|  D'(s)\right\|_{\L^2}^2\,\d s
		+ 2\int_0^t\left\langle
		b(V(s)+  D(s))\,,D(s)\right\rangle_{\L^2}\,\d s,
		\label{2.12}
	\end{align}
thanks to integration by parts (in fact, the second equality is formal, since $D''(s)$ may not belong to $\L^2$); the equality of the third line with the first is obtained by smoothing the term $b\left( v_R(t)+d_R(t)\right)$ in \eqref{d} and passing to the limit using the continuity property in \cite[Lemma 1.2]{pardoux}).
	
	As in \cite{FangZhang}, we consider the Lyapunov function,
	\[
		\Phi(r) := \exp\left( \int_0^r \frac{\d z}{1+z\log_+z}\right)
		\qquad[r>0].
	\]
	Owing to \eqref{2.12} and a second application of the chain rule,
	\begin{equation}\label{2.13}\begin{split}
		\Phi\left( \|D(t)\|_{\L^2}^2\right)
			&= \Phi\left( \|u_0\|_{\L^2}^2\right)- \int_0^t
			\Phi'\left( \|D(s)\|_{\L^2}^2\right)
			\left\| D'(s) \right\|_{\L^2}^2\,\d s\\
		&\qquad+ 2 \int_0^t\Phi'\left( \|D(s)\|_{\L^2}^2\right)
			\left\langle b(V(s)+D(s)) \,,D(s)\right\rangle_{\L^2}\,\d s.
	\end{split}\end{equation}
	
	Define
	\[
		C_b := \sup_{z\in\R}\frac{|b(z)|}{1+|z|\log_+|z|},
	\]
	and observe that $C_b\in(0\,,\infty)$, thanks to the $L\log L$
	growth of $b$. Moreover,
	\begin{align*}
		&\left\langle b(V(s)+D(s)) \,,D(s)\right\rangle_{\L^2}\\
		&\qquad \le C_b\int_0^1 \left[1+ \left\{ |V(s\,,x)| + |D(s\,,x)|\right\}
			\log_+\left( |V(s\,,x)|+|D(s\,,x)|\right)\right] \times
			|D(s\,,x)|\,\d x\\
		&\qquad \le C_b\int_0^1 \left[1+\left\{ M + |D(s\,,x)|\right\}
			\log_+\left( M +|D(s\,,x)|\right)\right]\times
			|D(s\,,x)|\,\d x,
	\end{align*}
	for every $s\in[0\,,T]$. In particular,
	\[
		\left\langle b(V(s)+D(s)) \,,D(s)\right\rangle_{\L^2}
		\apprle \left\| D(s)\right\|_{\L^2\log\L}^2
		+ \| D(s)\|_{\L^2}^2 + \|D(s)\|_{\L^1},
	\]
	for all $s\in[0\,,T]$, where the implied constant depends only on $(C_b\,,M)$.
	The Cauchy--Schwarz inequality yields
	\[
		\|D(s)\|_{\L^1} \le \|D(s)\|_{\L^2}
		\apprle \|D(s)\|_{\L^2}^2+1
		\qquad\text{for all $s>0$},
	\]
	and hence
	\[
		\left\langle b(V(s)+D(s)) \,,D(s)\right\rangle_{\L^2}
		\le \bar{C}\left\{\left\| D(s)\right\|_{\L^2\log\L}^2
		+ \| D(s)\|_{\L^2}^2 + 1\right\},
	\]
	uniformly for all $s\in[0\,,T]$, where $\bar{C}$ is a non-random
	and finite constant that depends only on $(C_b\,,M)$. Recall
	that $D$ is differentiable in $x$, $D'(s)\in \L^2$ a.s.\ for all $s>0$, 
	and $D$ satisfies the Dirichlet boundary condition.
	Therefore, we may apply the logarithmic Sobolev inequality
	[Theorem \ref{th:Gross}] with $\varepsilon := 1/(2\bar{C})$ in order to
	see that
	\[
		\left\langle b(V(s)+ D(s)) \,,D(s)\right\rangle_{\L^2}
		\le  \tfrac12 \| D'(s) \|_{\L^2}^2 + c_*\left\{\|D(s)\|_{\L^2}^2
		+ \|D(s)\|_{\L^2}^2 \log_+\left( \|D(s)\|_{\L^2}^2\right)+1\right\},
	\]
	uniformly for all $s\in[0\,,T]$, where $c_*$ is a non-random
	and finite constant, and depends only on $(C_b\,,M)$. Since
	$a\log_+a + 1 \ge a$ for all $a\ge 0$ and because
	the coefficient of $\|D'(s)\|_{\L^2}^2$ is one-half in the
	above displayed inequality, we can deduce
	the following from \eqref{2.13}:
	\[
		\Phi\left( \|D(t)\|_{\L^2}^2\right)
		\apprle \Phi\left( \|u_0\|_{\L^2}^2\right) + \int_0^t\Phi'\left( \|D(s)\|_{\L^2}^2\right)
		\left\{1+\|D(s)\|_{\L^2}^2 \log_+\left( \|D(s)\|_{\L^2}^2\right) \right\}\d s,
	\]
	uniformly for all $t\in[0\,,T]$, where the implied constant is non-random
	and finite, and depends only on $(C_b\,,M)$.
	But $\Phi'(r)[1+r\log_+r]=\Phi(r)$ for all $r\ge0$. Therefore, the preceding
	inequality implies that
	\begin{equation}\label{e_gronwall}
		\Phi\left( \|D(t)\|_{\L^2}^2\right)
		\apprle \Phi\left( \|u_0\|_{\L^2}^2\right) +  
		\int_0^t\Phi\left( \|D(s)\|_{\L^2}^2\right)\d s,
	\end{equation}
	uniformly for all $t\in[0\,,T]$, where the implied constant is non-random
	and finite, and depends only on $(C_b\,,M\,,T)$, but not on $R$.
	Thanks to the definitions of $\tau(R)$ and $\tau_M(R)$, we  know already
	that $\|D(t)\|_{\L^2} \leq R+M < \infty$ for all $t \in [0\,,T]$. 
	It then follows  from \eqref{e_gronwall} and Gronwall's inequality that
	$\sup_{t\in[0,T]}\Phi(\|D(t)\|_{\L^2}^2)$ is a.s.\ bounded from above by a
	non-random finite number $B(C_b\,,M\,,T)$, that depends only 
	on $(C_b\,,M\,,T)$ (but not on $R$), whence
	\begin{equation}\label{2.20}
		\sup_{R>0}
		\E\left[\Phi\left( \|d_R(T\wedge\tau(R)\wedge\tau_M(R)\|_{\L^2}^2
		\right)\right]   \le B(C_b\,,M\,,T).
	\end{equation}
	
	Next we observe that, almost surely on the event
	$\{\tau(R)<\tau\wedge T \le \tau_M(R)\}$,
	\begin{align*}
		\left\|d_R(T\wedge\tau(R)\wedge\tau_M(R))\right\|_{\L^2}
			& = \|d_R(\tau(R)) \|_{\L^2} \\
		&= \left\|
			u(\tau(R)) - v_R(\tau(R)) \right\|_{\L^2}\\
		&\ge \|u(\tau(R))\|_{\L^2} - \|v_R(\tau(R))\|_{\L^2}\\
		&= R -  \|v_R(\tau(R))\|_{\L^2}.
	\end{align*}
	On the other hand,
	\begin{align*}
		\|v_R(\tau(R))\|_{\L^2} &\le \sup_{t\in[0,\tau \wedge T)}\sup_{x\in[0,1]}
			|v_R(t\,,x)|\\
		&\le M\qquad\text{a.s.\ on $\{\tau(R) < \tau \wedge T\le\tau_M(R)\}$}.
	\end{align*}
	Thus, we see that
	\[
		\left\|d_R(T\wedge\tau(R)\wedge\tau_M(R))\right\|_{\L^2}
		\ge R-M\qquad\text{a.s.\ on $\{\tau(R)< \tau \wedge T\le \tau_M(R)\}$},
	\]
	whence
	\[
		\Phi\left(
		\left\|d(T\wedge\tau(R)\wedge\tau_M(R))\right\|_{\L^2}^2
		\right)
		\ge \Phi\left((R-M)^2\right)
		\quad\text{a.s.\ on $\{\tau(R)< \tau \wedge T\le \tau_M(R)\}$},
	\]
	as long as $R>M$.
	Combine this with \eqref{2.20} to see that
	\[
		\P\left\{ \tau(R)< \tau \wedge T\le \tau_M(R)\right\}
		\le \frac{B(C_b\,,M\,,T)}{\Phi\left((R-M)^2\right)}\qquad\text{for all $R>M>0$}.
	\]
	The preceding inequality and \eqref{2.11} together show that
	\[
		\P\left\{\tau(R)< \tau \wedge T\right\} \le
		\frac{B(C_b\,,M\,,T)}{\Phi\left((R-M)^2\right)} + \frac{%
		A_T}{M},
	\]
	for all $R>M$. We first let $R\to\infty$ and then $M\to\infty$ in order
	to deduce \eqref{e_suff}. This completes the proof of Theorem \ref{th:L2}.
\end{proof}

\section{Prelude to the proof of Theorem \ref{th:C:alpha}}
Throughout this section, we consider \eqref{SHE}
only in the classical case where
\[
	\text{$b$ and $\sigma$ are globally Lipschitz continuous.}
\]
It is well known that, in this case, \eqref{SHE}
is well-posed (see Walsh \cite[Ch.\ 3]{Walsh}).
Here, we develop some {\it a priori} moment
bounds. One of our main goals is to establish {\it a priori}
smoothness bounds for the solution
of \eqref{SHE} that are valid up to and including the boundary of $[0\,,1]$.
This endeavor requires some careful estimates and ultimately leads to an interesting
optimal regularity theorem [Theorem \ref{th:OR}] that forms one of the main ingredients
in the proof of Theorem \ref{th:C:alpha}.



\subsection{Moment Bounds}

We begin by establishing moment bounds for the solution
$u$ to \eqref{SHE}.
With this goal in mind, we frequently use the
elementary fact that for every globally Lipschitz function $f:\R\to\R$, there are constants $c(f)$ and $\Lip(f)$ such that
\begin{equation}\label{Lip(f)}
	|f(z)| \le c(f) +  \Lip(f)|z|,\qquad\text{for all $z\in\R$}.
\end{equation}
One possibility is to take $c(f) = \vert f(0)\vert$ and $\Lip(f) =$~Lip$(f)$, but often, $\Lip(f)$ can be chosen strictly smaller than Lip$(f)$. We will only consider the case where
\begin{equation}\label{LipLip}
	\Lip(b)\ge 4\Lip(\sigma)^4>0.
\end{equation}
The significance of this assumption, which is not a restriction since $\Lip(b)$ can be chosen arbitrarily large, will manifest itself later on in the proof
of Theorem \ref{th:C:alpha}.

{\em	Throughout this section, \eqref{LipLip}
	will be assumed tacitly.
}

\begin{proposition}\label{I:pr:Lk}
	There exists a finite universal constant $A$ such that
	\[
		\sup_{x\in[0,1]}
		\E\left(|u(t\,,x)|^k\right)\le
		A^k\left(\|u_0\|_{\L^\infty} + \frac{c(b)}{\Lip(b)} +
		\frac{c(\sigma)}{\sqrt{\Lip(\sigma)}}\right)^k\cdot
		\exp\left(Ak\Lip(b) t\right),
	\]
	uniformly for all real numbers $t\ge0$ and $k\in[2\,, \sqrt{\Lip(b)}/\Lip(\sigma)^2]$.
\end{proposition}

\begin{remark}
	Condition \eqref{LipLip} ensures that the interval
	$[2\,,\sqrt{\Lip(b)}/\Lip(\sigma)^2]$ is nonempty.
\end{remark}

\begin{proof}
	Throughout, we write $\cB(t)$ for the box
	\begin{equation}\label{R(t)}
		\cB(t) := (0\,,t)\times(0\,,1),\qquad\text{for every $t\ge 0$}.
	\end{equation}
	In light of Definition \ref{def1.7}, 
	\begin{equation}\label{I:mild}\begin{split}
		&u(t\,,x) \\
		&= (\cG_tu_0)(x) + \int_{\cB(t)}
			G_{t-s}(x\,,y)b(u(s\,,y))\,\d s\,\d y
			+ \int_{\cB(t)} G_{t-s}(x\,,y)
			\sigma(u(s\,,y))\,\xi(\d s\,\d y),
	\end{split}\end{equation}
	where $\cG$ and $G$ were defined respectively in
	\eqref{G_t} and \eqref{G}. 
	The solution $u$ to \eqref{SHE} has a [jointly] continuous version 
	which is the unique continuous solution of \eqref{I:mild}.
	
	Let us also recall that one verifies the existence of a solution to 
	\eqref{I:mild} by applying Picard's iteration method as follows: Set
	$u_0(t\,,x) := u_0(x)$ for all $x\in[0\,,1]$, and then iteratively define
	\begin{align*}
		&u_{n+1}(t\,,x)\\
		&:=(\cG_tu_0)(x) + \int_{\cB(t)}
			G_{t-s}(x\,,y)b(u_n(s\,,y))\,\d s\,\d y
			+ \int_{\cB(t)} G_{t-s}(x\,,y)
			\sigma(u_n(s\,,y))\,\xi(\d s\,\d y).
	\end{align*}
	Then, 
	\begin{equation}\label{unu}
		\lim_{n\to\infty} u_n(t\,,x) = u(t\,,x)
		\qquad\text{in $L^k(\Omega)$ for all $k\ge 1$, $t\ge0$, and $x\in[0\,,1]$}.
	\end{equation}
	See Walsh \cite[Ch.\ 3]{Walsh}.
	
	We now follow Foondun and Khoshnevisan
	\cite{FK} and consider a two-parameter family $\{\cN_{\beta,k}\}_{\beta>0,k\ge1}$
	of norms---each defined on the space of space-time random fields---as follows:
	For all real numbers $\beta>0$ and $k\ge1$, and for every
	space-time random field $\Phi:=\{\Phi(t\,,x);\, t\ge0, x\in[0\,,1]\}$,
	\begin{equation}\label{N}
		\cN_{\beta,k} (\Phi) := \sup_{t\ge0}\sup_{x\in[0,1]}\left(
		\e^{-\beta t}\| \Phi(t\,,x)\|_k\right).
	\end{equation}
	Since $|(\cG_tu_0)(x)|\le \|u_0\|_{\L^\infty}$ for all $x\in[0\,,1]$
	and $t> 0$, we can  write
	\begin{equation}\label{I:NT1T2}
		\cN_{\beta,k}(u_{n+1}) \le \|u_0\|_{\L^\infty} +
		\cT_1 + \cT_2,
	\end{equation}
	where
	\begin{align*}
		\cT_1 &:= \sup_{t\ge0}\sup_{x\in[0,1]}\left(
			\e^{-\beta t}\int_{\cB(t)} G_{t-s}(x\,,y)\|b(u_n(s\,,y))\|_k\,
			\d s\,\d y\right),\\
		\cT_2 &:=\sup_{t\ge0}\sup_{x\in[0,1]}\left(
			\e^{-\beta t}\left\|\int_{\cB(t)} G_{t-s}(x\,,y)
			\sigma(u_n(s\,,y))\, \xi(\d s\,\d y)\right\|_k\right).
	\end{align*}
	We plan to estimate $\cT_1$ and $\cT_2$ in this order.
	
	Recall that $\int_0^1 G_\rho(x\,,y)\,\d y$ is the probability that
	Brownian motion, started at $x\in[0\,,1]$, does not reach the set
	$\{0\,,1\}$ before time $\rho>0$. Therefore, the support theorem for the
	Wiener measure implies that
	\begin{equation}\label{int:G}
		\int_0^1 G_\rho(x\,,y)\,\d y < 1
		\qquad\text{for all $\rho>0$ and
		$x\in[0\,,1]$.}
	\end{equation}
	From this and \eqref{Lip(f)} we can deduce that
	\begin{align*}
		\cT_1 &\le c(b)\sup_{t\ge0}\left(t\e^{-\beta t}\right)
			+\Lip(b)\sup_{t\ge0}\sup_{x\in[0,1]}\left(\e^{-\beta t}
			\int_{\cB(t)} G_{t-s}(x\,,y) \|u_n(s\,,y)\|_k\,\d s\,\d y\right)\\
		&= \frac{c(b)}{\e\beta} + \Lip(b)
			\sup_{t\ge0}\sup_{x\in[0,1]}\left(
			\int_{\cB(t)} \e^{-\beta (t-s)}G_{t-s}(x\,,y)
			\e^{-\beta s}\|u_n(s\,,y)\|_k\,\d s\,\d y\right)\\
		&\le \frac{c(b)}{\beta} + \Lip(b)\cN_{\beta,k}(u_n)\cdot
			\sup_{t\ge0}\sup_{x\in[0,1]}\left(
			\int_{\cB(t)} \e^{-\beta (t-s)}G_{t-s}(x\,,y) \,\d s\,\d y\right).
	\end{align*}
	Another appeal to \eqref{int:G} yields the following inequality, which
	is our desired bound for the quantity $\cT_1$:
	\begin{equation}\label{I:T1}\begin{split}
		\cT_1 &\le \frac{c(b)}{\beta} + \Lip(b)\cN_{\beta,k}(u_n)\cdot
			\sup_{t\ge0}\left(
			\int_0^t \e^{-\beta (t-s)}\,\d s\right)\\
		&\le \frac{c(b)+ \Lip(b)\cN_{\beta,k}(u_n)}{\beta}.
	\end{split}\end{equation}
	
	In order to estimate $\cT_2$, we first recall that
	\[
		\left\|\int_{\cB(t)} G_{t-s}(x\,,y)
		\sigma(u_n(s\,,y))\, \xi(\d s\,\d y)\right\|_k^2 \le
		4k\int_{\cB(t)}
		\left[G_{t-s}(x\,,y)\right]^2\|\sigma(u_n(s\,,y))\|_k^2\,\d s\,\d y,
	\]
	thanks to a suitable application
	of the BDG inequality (see \cite[Proposition 4.4, p.\ 36]{Kh:CBMS}).
	An appeal to \eqref{Lip(f)} yields
	\[
		\|\sigma(u_n(s\,,y))\|_k \le c(\sigma) +
		\Lip(\sigma) \|u_n(s\,,y)\|_k \le c(\sigma) +
		\Lip(\sigma) \e^{\beta s}\cN_{\beta,k}(u_n).
	\]
	Thus, we see that
	\begin{align*}
		&\left\|\int_{\cB(t)} G_{t-s}(x\,,y)
			\sigma(u_n(s\,,y))\, \xi(\d s\,\d y)\right\|_k^2 \\
		&\le  4k\int_{\cB(t)}
			\left[G_{t-s}(x\,,y)\right]^2
			\left(c(\sigma) +
			\Lip(\sigma) \e^{\beta s}\cN_{\beta,k}(u_n)\right)^2\,\d s\,\d y\\
		&\le 8 kc(\sigma)^2\int_{\cB(t)}
			\left[G_s(x\,,y)\right]^2\,\d s\,\d y + 8k[\Lip(\sigma)]^2
			\left[\cN_{\beta,k}(u_n)\right]^2
			\int_{\cB(t)} \e^{2\beta s}
			\left[G_{t-s}(x\,,y)\right]^2\,\d s\, \d y.
	\end{align*}
	Next we observe that, uniformly for all $t,\beta>0$,
	\[
		\int_{\cB(t)}\left[G_s(x\,,y)\right]^2\,\d s\,\d y
		\le \e^{2\beta t}\int_{\cB(t)} \e^{-2\beta s}[G_s(x\,,y)]^2\,\d s\,\d y
		\apprle \frac{\e^{2\beta t}}{\sqrt\beta},
	\]
	where the final bound is justified by Lemma \ref{lem:A2}, with the implied
	universal constant being equal to $(\pi\sqrt{2})^{-1}$.
	Similarly,
	\[
	   \e^{-2\beta t} \int_{\cB(t)} \e^{2\beta s} [G_{t-s}(x\,,y)]^2 \,\d s\,\d y 
	   = \int_{\cB(t)} \e^{-2\beta s} [G_s(x\,,y)]^2 \,\d s\,\d y 
	   \apprle \frac{1}{\sqrt{\beta}},
	\]
	uniformly for all $t,\beta>0$.
	Consequently,
	\[
		\e^{-2\beta t}\left\|\int_{\cB(t)} G_{t-s}(x\,,y)
		\sigma(u_n(s\,,y))\, \xi(\d s\,\d y)\right\|_k^2
		\apprle \frac{kc(\sigma)^2  + k[\Lip(\sigma)]^2
		\left[\cN_{\beta,k}(u_n)\right]^2}{\sqrt\beta},
	\]
	uniformly for all $n\ge0$, $x\in[0\,,1]$, $\beta>0$, and $k\ge 2$.
	We take square roots of both sides, then optimize over $t\ge0$
	and $x\in[0\,,1]$ in order to see that
	\begin{equation}\label{I:T2}
		\cT_2 \apprle\frac{k^{1/2}}{\beta^{1/4}}\cdot\left(
		c(\sigma)  + \Lip(\sigma)
		\cN_{\beta,k}(u_n)\right),
	\end{equation}
	with the same uniformity properties as before on $(k\,,\beta\,,x\,,n)$.
	This is the desired inequality for $\cT_2$.
	
	   We now combine
	\eqref{I:NT1T2} with \eqref{I:T1} and \eqref{I:T2} in order to see that
	\begin{equation}\label{I:NKLN}
		\cN_{\beta,k}(u_{n+1}) \le K_{\beta,k} + L_{\beta,k}
		\cN_{\beta,k}(u_n),
	\end{equation}
	uniformly for all $\beta>0$, $k\ge 2$, and $n\ge 0$, where
	\begin{align*}
		K_{\beta,k} &:= c\left(\|u_0\|_{\L^\infty} +
			\frac{c(b)}{\beta} + \frac{k^{1/2}c(\sigma)}{\beta^{1/4}}\right),
			\\
		L_{\beta,k} &:= c\max\left(
			\frac{\Lip(b)}{\beta} ~,~ \frac{k^{1/2}\Lip(\sigma)}{\beta^{1/4}}\right),
	\end{align*}
	for a sufficiently-large finite universal constant $c>1$.
	Let us choose $\beta:=16c^4\Lip(b)$ and observe that, for this choice
	of $\beta$,
	\begin{align*}
		K_{16c^4\Lip(b),k} &\le c\|u_0\|_{\L^\infty} + \frac{c(b)}{16\Lip(b)} +
			\frac{k^{1/2}c(\sigma)}{2[\Lip(b)]^{1/4}},\\
		L_{16c^4\Lip(b),k} &\le \max\left( \frac{1}{16}\,,\frac{k^{1/2}\Lip(\sigma)}{%
			2[\Lip(b)]^{1/4}}\right).
	\end{align*}
	In this way, we may simplify \eqref{I:NKLN} to the following recursive
	inequality: Uniformly for all integers $n\ge 0$ and real numbers
	$k\in[2\,, \sqrt{\Lip(b)}/\Lip(\sigma)^2]$,
	\[
		\cN_{16c^4\Lip(b),k}(u_{n+1}) \le c\|u_0\|_{\L^\infty} + \frac{c(b)}{16\Lip(b)} +
		\frac{c(\sigma)}{2\Lip(\sigma)} + \tfrac12
		\cN_{16c^4\Lip(b),k}(u_n).
	\]
	Since $\cN_{16c^4\Lip(b),k}(u_0)=\|u_0\|_{\L^\infty}$ is finite, the preceding
	implies that 
	$\sup_{n\ge0}\cN_{\beta,k}(u_n)<\infty$ 
	for all real numbers $k\in[2\,, \sqrt{\Lip(b)}/\Lip(\sigma)^2]$,
	and, more significantly,
	\[
		\limsup_{n\to\infty}\cN_{16c^4\Lip(b),k}(u_n) \le
		2c\|u_0\|_{\L^\infty} + \frac{c(b)}{8\Lip(b)} +
		\frac{c(\sigma)}{\Lip(\sigma)}.
	\]
	By Fatou's lemma and \eqref{unu},
	\begin{equation}\label{I:N}\begin{split}
		\cN_{16c^4\Lip(b),k}(u) &\le \limsup_{n\to\infty}\cN_{16c^4\Lip(b),k}(u_n)\\
		&\apprle \|u_0\|_{\L^\infty} + \frac{c(b)}{\Lip(b)} +
			\frac{c(\sigma)}{\Lip(\sigma)},
	\end{split}\end{equation}
	uniformly for all $k\in[2\,, \sqrt{\Lip(b)}/\Lip(\sigma)^2]$.
	We can unscramble this inequality directly,
	using only \eqref{N}, in order to deduce the proposition.
\end{proof}

In the context of Proposition \ref{I:pr:Lk},
one might wonder about the moments of order $k$ when
$k>\sqrt{\Lip(b)}/\Lip(\sigma)^2$. In that case, it is possible
to adjust only slightly the proof of Proposition \ref{I:pr:Lk}
in order to obtain the following.

\begin{proposition}\label{I:pr:Lk:bis}
	If $\Lip(b)\ge 4\Lip(\sigma)^4 >0$, then
	there exists a finite universal constant $A$ such that
	\[
		\sup_{x\in[0,1]}
		\E\left(|u(t\,,x)|^k\right)\le
		A^k\left(\|u_0\|_{\L^\infty} + \frac{c(b)}{[\Lip(\sigma)]^4} +
		\frac{c(\sigma)}{\Lip(\sigma)}\right)^k\cdot
		\exp\left(Ak^3[\Lip(\sigma)]^4 t\right),
	\]
	uniformly for all real numbers $t\ge0$ and $k>\sqrt{\Lip(b)}/\Lip(\sigma)^2$.
\end{proposition}

\begin{proof}
	We follow the proof of Proposition \ref{I:pr:Lk} up to and including
	\eqref{I:NKLN} without change. However, if $k>\sqrt{\Lip(b)}/\Lip(\sigma)^2$,
	then we choose the auxilliary parameter $\beta$ slightly differently. Namely, let
	us define
	\[
		\beta := 16c^4k^2[\Lip(\sigma)]^4,
	\]
	notation being that of \eqref{I:NKLN}. For this particular choice,
	\begin{equation}\label{rde3.13}
		K_{\beta,k} = c\|u_0\|_{\L^\infty} +
		\frac{c(b)}{16c^3 k^2[\Lip(\sigma)]^4} + \frac{c(\sigma)}{2\Lip(\sigma)},
	\end{equation}
	and $L_{\beta,k} = \nicefrac12$. We apply the preceding particular choice of
	$\beta$ in \eqref{I:NKLN} in order to see that
	\[
		\cN_{\beta,k}(u_{n+1})\le K_{\beta,k} + \tfrac12\cN_{\beta,k}(u_n),
	\]
	for all $n\ge0$. This shows in particular that
	$\cN_{\beta,k}(u_n)\apprle K_{\beta,k}$ uniformly for all $n\ge0$,
	which is another way to state the result.
\end{proof}

\subsection{An optimal regularity theorem}

Next we derive the following optimal regularity result.

\begin{theorem}\label{th:OR}
	The following logical implications are valid:
	\[
		\alpha \in (0\,,\tfrac12),\ u_0\in \C^\alpha_0
		\quad\Longrightarrow\quad
		\P\left\{ u(t)\in \C^\alpha_0
		\text{ \  for all $t>0$}\right\}=1,
	\]
	and
	\[
		\alpha \in [\tfrac12\,,1],\	u_0\in \C^\alpha_0
		\quad\Longrightarrow\quad
		\P\left\{ u(t)\in \bigcap_{\varepsilon >0}\C^{\frac{1}{2}-\varepsilon}_0
		\text{ for all $t>0$}\right\}=1.
	\]
\end{theorem}

We begin by establishing some quantitive estimates that describe
the smoothness properties of the solution to \eqref{SHE}. Clearly,
this work prepares for Theorem \ref{th:OR}, since among
other things, Theorem \ref{th:OR} asserts that the solution to
\eqref{SHE} is H\"older continuous.

Let us first observe that \eqref{G_t} identifies every kernel $G_t$ with a linear operator
$\cG_t$ in the usual way.
It is well known---and easy to verify directly using \eqref{G_t}---that
$\{\cG_t\}_{t\ge0}$ is a semigroup
of linear operators that is bounded in $\L^\infty$ and
$(\cG_t\1)(x)\le 1$ for all $x\in[0\,,1]$ and $t\ge0$, where
$\1(x):=x$ for all $x\in[0\,,1]$.

The semigroup $\{\cG_t\}_{t\ge0}$ is said to be \emph{Feller
uniformly on a function class $\F$} if
\[
	\lim_{t\downarrow0} \sup_{f\in\F}\sup_{x\in[0,1]} \left|(\cG_tf)(x)-f(x)\right|=0.
\]
The following lemma shows that
$\{\cG_t\}_{t\ge0}$ is indeed  Feller uniformly on every bounded subset $\F$ of
$\C^\alpha_0$
for every $\alpha\in(0\,,1]$.
In fact, the following contains the quantitative improvement,
\[
	\sup_{f\in\F}\sup_{x\in[0,1]}|(\cG_tf)(x)-f(x)| =
	O\left(t^{\alpha/2}\right)\qquad
	\text{as $t\downarrow0$},
\]
for every bounded subset $\F$ of the Banach space $\C^\alpha_0$.

\begin{lemma}[A quantitative Feller property]\label{lem:Feller}
	Choose and fix $\alpha\in(0\,,1]$. Then
	\[
		\sup_{x\in[0,1]}\sup_{t\ge 0}\frac{\left|
		(\cG_tf)(x) - f(x)\right|}{t^{\alpha/2}}\apprle
		\|f\|_{\C^\alpha_0},
	\]
	uniformly for all $f\in\C^\alpha_0$.
\end{lemma}

\begin{proof}
	Let us choose and fix $t>0$ and $x\in[0\,,1]$. Then
	\begin{align*}
		\left| (\cG_tf)(x) - f(x)\right| &= 
			\left| \int_0^1 G_t(x\,,z) f(z)\, \d z - f(x) \right| \\
		& = \left| \int_0^1 G_t(x\,,z) [f(z) - f(x)]\, \d z + f(x)
			\left(\int_0^1 G_t(x\,,z) \,\d z - 1 \right) \right| \\
		& \le   \int_0^1 G_t(x\,,z) \|f\|_{\C^\alpha_0}
			\vert z - x \vert^\alpha\, \d z + \vert f(x)\vert
			\left[ 1 - \int_0^1 G_t(x\,,z)\, \d z \right].
	\end{align*}
	Use the inequality \eqref{boundG} for the first term, and the fact that
	\[
		\vert f(x)\vert = \vert f(x) - f(0) \vert = \vert f(x) - f(1) \vert
		\le\|f\|_{\C^\alpha_0}(\min(x\,,1-x))^\alpha,
	\]
	for the second term, in order to see that
	\[
		\left| (\cG_tf)(x) - f(x)\right| \le c_0 \|f\|_{\C^\alpha_0}\,
		t^{\alpha/2} + \|f\|_{\C^\alpha_0}\, (\min(x,1-x))^\alpha
		\left( 1 - \int_0^1 G_t(x\,,z) \,\d z \right),
	\]
	where $c_0$ is a universal constant.
	
	  Let $B:=\{B_t\}_{t\ge0}$ denote a standard
	  1-dimensional Brownian motion, and consider the [a.s.\ finite] stopping time
	\[
			\tau := \inf\left\{t>0:\, B_t\in\{0\,,1\}\right\}.
	\]
	Thus, we may write, using standard notation: For all $x\in[0\,,1]$ and $t>0$,
	\begin{align*}
		1 - \int_0^1 G_t(x\,,z)\, \d z = \P_x \{\tau \le t \}
		    &= \P_0 \left\{ \sup_{s \in [0,1]} \vert B_s\vert \geq
		    \frac{\min(x\,,1-x)}{t^{1/2}}\right\} \\
		& \leq \exp\left[- \frac{(\min(x\,,1-x))^2}{2t} \right].
	\end{align*}
	Let $y := \min(x\,,1-x)$. By the last inequality,
	\begin{align*}
		y^\alpha \left( 1 - \int_0^1 G_t(x\,,z)\, \d z \right)
		\le y^\alpha \exp\left[- \frac{y^2}{2t}\right]
		\le t^{\alpha/2} \left(\frac{y^2}{t} \right)^{\alpha/2}
		\exp\left[- \frac{y^2}{2t}\right] \le c t^{\alpha/2}.
	\end{align*}
	We conclude that
	\[
		\left| (\cG_tf)(x) - f(x)\right| \le \tilde c_\alpha\,
		\|f\|_{\C^\alpha_0}\, t^{\alpha/2},
	\]
	where $\tilde c_\alpha = 1 + \sup_{y \ge 0} (y^{\alpha/2}\e^{-y/2})=
	1+(\alpha/\e)^{\alpha/2}.$
\end{proof}

\begin{remark}
	Suppose that $f(y) = \sum_{n=1}^\infty f_n \sin(n\pi y)$
	for $y\in[0\,,1]$, where
	the Fourier sine coefficients $\{f_n\}_{n=1}^\infty$ of $f$
	satisfy $\|f\|_{1,\alpha}:=
	\sum_{n=1}^\infty \vert f_n \vert n^\alpha < \infty$
	for some $\alpha\in(0\,,1]$. Then,
	it is not hard to see that $f\in\C_0^\alpha$
	and $\|f\|_{\C_0^\alpha}\le\pi^\alpha\|f\|_{1,\alpha}$. Indeed, $f$ vanishes
	on $\{0\,,1\}$, and for every distinct $x,y\in[0\,,1]$,
	\[
		\frac{|f(x)-f(y)|}{|x-y|^\alpha}
		\le \sum_{n=1}^\infty|f_n|  \cdot\frac{|\sin(n\pi x)-\sin(n\pi y)|}{|x-y|^\alpha}
		\le \pi^\alpha\|f\|_{1,\alpha}.
	\]
	In this particular case,  a simpler argument than the proof of Lemma \ref{lem:Feller} yields
	the slightly weaker bound,
	\[
		\sup_{t>0}\sup_{x\in[0,1]}\frac{\left| (\cG_tf)(x) - f(x)\right|}{t^{\alpha/2}}
		\apprle \|f\|_{1,\alpha}.
	\]
	Indeed, $(\cG_tf)(x) = \int_0^1 G_t(x\,,y) f(y) \,\d y
	= \sum_{n=1}^\infty f_n \sin(n\pi x) \exp(- n^2 \pi^2 t/2),$ and
	so
	\begin{align*}
		\left| (\cG_tf)(x) - f(x)\right|
			&\le  \left| \sum_{n=1}^\infty f_n \sin(n\pi x)
			\left(1 - \exp(- n^2 \pi^2 t/2) \right) \right| \\
		&\le \frac{\pi^\alpha t^{\alpha/2}}{2^{\alpha/2}}
			\sum_{n=1}^\infty |f_n| n^\alpha
		   	\left|\frac{1 - \exp(- n^2 \pi^2 t/2)}{%
			(n^2\pi^2t/2)^{\alpha/2}}  \right| \\
		&\le \frac{C_\alpha\pi^\alpha}{2^{\alpha/2}}\cdot\|f\|_{1,\alpha}\,
			 t^{\alpha/2},
	\end{align*}
	where $C_\alpha = \sup_{y > 0} y^{-\alpha/2} (1 - \exp(-y)) < \infty$.
\end{remark}

The next result is also a deterministic lemma. Among other things,
it asserts that every $\cG_t$ maps each $\C^\alpha_0$ boundedly to
$\C^\alpha_0$.

\begin{lemma}\label{lem:Feller:2}
	Choose and fix an arbitrary $\alpha\in(0\,,1]$.
	Then,
	\[
		\sup_{0\le x<x'\le 1}\sup_{t\ge 0}
		\frac{|(\cG_tf)(x) - (\cG_tf)(x') |}{|x-x'|^{\alpha}}
		\apprle \|f\|_{\C^\alpha_0},
	\]
	uniformly for all $f\in\C^\alpha_0$.
\end{lemma}

\begin{proof}
	It is well known that the Green's function $G$ can be represented as follows:
	\begin{equation}\label{G:phi}
		G_t(x\,,z) = \varphi_t(z-x) - \varphi_t(z+x)
		\qquad\text{for all $t>0$ and $x,z\in(0\,,1)$,}
	\end{equation}
	where
	\[
		\varphi_t(x) := \frac{1}{\sqrt{2\pi t}} \sum_{n=-\infty}^\infty
		\exp\left[- \frac{(x-n)^2}{2t} \right].
	\]
	See, for example, Bally et al \cite[proof of Lemma A2]{BallyMilletSanz}.
	
	Choose and fix $x,y\in[0\,,1]$
	such that
	\[
		h:=y-x>0.
	\]
	Thanks to \eqref{G:phi},
	\begin{align*}
	   &(\cG_tf)(x) - (\cG_tf)(y)\\
		 &\qquad= \int_0^1 \left[G_t(x\,,z) - G_t(y\,,z)\right] f(z)\, \d z \\
		   &\qquad= \int_0^1 \left[ \varphi_t(z-x) - \varphi_t(z-h-x)\right] f(z)\, \d z -
		   \int_0^1 \left[ \varphi_t(z+x) - \varphi_t(z+h+x)\right] f(z)\, \d z,
	\end{align*}
	for all $t>0$.
	We follow Bally et al \cite{BallyMilletSanz}
	and organize the preceding as follows:
	\[
		(\cG_tf)(x) - (\cG_tf)(y)	= \cI_1 - \cI_2 + \cI_3 - \cI_4 - \cI_5 + \cI_6,
	\]
	where:
	\begin{gather*}
		\cI_1 = \int_0^{1-h} \varphi_t(z-x) [f(z) - f(z+h)]\, \d z;\quad
		\cI_2 = \int_h^1 \varphi_t(z+x) [f(z) - f(z-h)]\, \d z;\\
		\cI_3 =\int_{1-h}^1 \varphi_t(z-x) f(z)\, \d z;\quad
		\cI_4 =\int_{-h}^0 \varphi_t(z-x) f(z+h) \,\d z;\quad
		\cI_5 = \int_0^h \varphi_t(z+x) f(z) \,\d z;\\
		\cI_6 =\int_1^{1+h} \varphi_t(z+x) f(z-h)\, \d z.
	\end{gather*}
	Since $f \in\C^\alpha_0$ and $\varphi_t$ is a probability density,
	$|\cI_\ell|\le h^\alpha\|f\|_{\C^\alpha_0}$ for $\ell=1,2$.
	Moreover, we can replace $f(z)$ by $f(z) - f(1)$
	in $\cI_3$ and $\cI_6$, and by $f(z) - f(0)$ in $\cI_4$ and $\cI_5$,
	in order to obtain
	\begin{align*}
		&\vert(\cG_tf)(x) - (\cG_tf)(y)\vert \\
		&\qquad\le 2h^\alpha\|f\|_{\C^\alpha_0}
			+  \|f\|_{\C^\alpha_0} \Bigg[ \int_{1-h}^1
			\varphi_t(z-x) \vert z - 1\vert^\alpha\, \d z + \int_{-h}^0
			\varphi_t(z-x) \vert z+h\vert ^\alpha\,\d z  \\
		&\qquad\hskip1.7in + \int_0^h \varphi_t(z+x) \vert z \vert^\alpha
			\,\d z + \int_1^{1+h} \varphi_t(z+x) \vert z - h - 1\vert^\alpha\, \d z \Bigg].
	\end{align*}
	Because the absolute values are all
	bounded above by $h$; the preceding quantity is at most
	$6h^\alpha\|f\|_{\C^\alpha_0}.$
	This completes the proof.
\end{proof}

We now begin to use the preceding two analytic results about the Dirichlet Laplacian,
acting on $\C^\alpha_0$, in order to derive smoothness results for the solution
to \eqref{SHE}. First, let us present a result about smoothness in the space variable.

Throughout, we write
\[
	u(t\,,x) = (\cG_t u_0)(x) + I(t\,,x),
\]
where
\[
	I(t\,,x) = \int_{\cB(t)}  G_{t-s}(x\,,y)	b(u(s\,,y))\,\d s\,\d y
	+ \int_{\cB(t)}  G_{t-s}(x\,,y) \sigma(u(s\,,y))\,\xi(\d s\,\d y),
\]
and $\cB(t)$ was defined in \eqref{R(t)}.
It might help to recall that \eqref{LipLip} is in place throughout the section.

\begin{proposition}\label{I:pr:u(x)-u(x')}
	Choose and fix $\alpha\in(0\,,1]$.
	There exists a finite universal constant $A$---independent of
	$(b\,,\sigma)$---such that
	\[
		\sup_{0\le x< x'\leq 1}
		\E\left( \left|\frac{I(t\,,x) - I(t\,,x')}{\vert x'-x\vert^{1/2}}
		\right|^k\right) \le A^k
		\left(	 k^{k/2} \cM_1^k + k^{k/2} \cM_2^k\, \cM_3^k\, \e^{Ak\Lip(b)t}\right),
	\]
	and
	\[
		\sup_{0\le x< x'\leq 1}
		\E\left( \left|\frac{u(t\,,x) - u(t\,,x')}{\vert x'-x\vert ^{\alpha\wedge(1/2)}}
		\right|^k\right) \le A^k
		\left( \|u_0\|_{\C^\alpha_0}^k +
		 k^{k/2} \cM_1^k + k^{k/2} \cM_2^k\, \cM_3^k\, \e^{Ak\Lip(b)t}\right),
	\]
	uniformly for all $u_0\in\C^\alpha_0$,
	$t\ge 0$, and $k\in[2\,, \sqrt{\Lip(b)}/\Lip(\sigma)^2]$,
	where:
	\[
		\cM_1 := c(b)+c(\sigma);\quad
		\cM_2 := \Lip(b)+\Lip(\sigma);\text{ and}\quad
		\cM_3 :=\|u_0\|_{\L^\infty} + \frac{c(b)}{\Lip(b)} +
		\frac{c(\sigma)}{\Lip(\sigma)}.
	\]
\end{proposition}

\begin{remark}\label{rem:u(x)-u(x')}
	Stated in other words, the above asserts
	that the moments of order $\le \sqrt{\Lip(b)}/\Lip(\sigma)^2$
	behave as those of a Gaussian random variable. Of course, such a
	statement can have nontrivial content only when $\Lip(b)\gg 4[\Lip(\sigma)]^4$.
\end{remark}

\begin{proof}
	Thanks to \eqref{I:mild}, we may write
	\begin{equation}\label{u-u:G-G:etc}
		\|u(t\,,x) - u(t\,,x') \|_k \le \left| (\cG_tu_0)(x) - (\cG_tu_0)(x')\right| +
		\| I(t\,,x) - I(t\,,x') \|_k ,
	\end{equation}
	and
	\begin{equation}\label{ITT}
			\| I(t\,,x) - I(t\,,x') \|_k \leq \|\cT_1\|_k + \|\cT_2\|_k,
		\end{equation}
	where
	\begin{align*}
		\cT_1 &:= \int_{\cB(t)} \left[ G_{t-s}(x\,,y)-G_{t-s}(x',y)\right]
			b(u(s\,,y))\,\d s\,\d y,\\
		\cT_2 &:= \int_{\cB(t)} \left[ G_{t-s}(x\,,y)-G_{t-s}(x',y)\right]
			\sigma(u(s\,,y))\,\xi(\d s\,\d y).
	\end{align*}
	
	Lemma \ref{lem:Feller:2} estimates the first term on the right-hand side of \eqref{u-u:G-G:etc}
	first as follows:
	\begin{equation}\label{II:G-G}
		\left| (\cG_tu_0)(x) - (\cG_tu_0)(x')\right|
		\apprle\|u_0\|_{\C^\alpha_0}\cdot |x-x'|^{\alpha},
	\end{equation}
	uniformly for all $t\ge 0$ and $x,x'\in[0\,,1]$.
	
	Next we estimate $\cT_1$. First,
	an appeal to \eqref{Lip(f)} yields
	\[
		\|\cT_1\|_k\le \int_{\cB(t)} \left| G_{t-s}(x\,,y)-G_{t-s}(x',y)\right|
		\left( c(b)+ \Lip(b)\|u(s\,,y)\|_k\right)\d s\,\d y.
	\]
	Lemma \ref{lem:A3} ensures that, for all $x,x'\in[0\,,1]$,
	\[
		\sup_{t>0}
		\int_{\cB(t)}\left| G_s(x\,,y)-G_s(x',y)\right|\d s\,\d y\apprle
		|x-x'|\log_+\left( \frac{1}{|x-x'|}\right),
	\]
	where $\log_+(a) := \log(\e\vee a)$ for all $a\in\R$.
	Furthermore,
	\begin{align*}
		\int_{\cB(t)}\left| G_{t-s}(x\,,y)-G_{t-s}(x',y)\right|&\cdot
			\|u(s\,,y)\|_k\,\d s\,\d y\\
		&\le \e^{\beta t}\cN_{\beta,k}(u)\sup_{t>0}
			\int_{\cB(t)}\left| G_{t-s}(x\,,y)-G_{t-s}(x',y)\right|\d s\,\d y\\
		&\apprle\e^{\beta t}\cN_{\beta,k}(u)\cdot |x-x'|\log_+
			\left(\frac{1}{|x-x'|}\right),
	\end{align*}
	thanks to a second appeal to Lemma \ref{lem:A3}.
	[The norm $\cN_{\beta,k}$ was defined in \eqref{N}.] It follows that
	\begin{equation}\label{II:T1}\begin{split}
		\|\cT_1\|_k &\le \left[
			c(b)+\Lip(b)\e^{\beta t}\cN_{\beta,k}(u)\right]
			\cdot |x-x'|\log_+\left(\frac{1}{|x-x'|}\right)\\
		&\apprle\left[
			c(b)+\Lip(b)\e^{\beta t}\cN_{\beta,k}(u)\right]\cdot
			|x-x'|^{1/2}.
	\end{split}\end{equation}
	The last line follows from the elementary fact that
	$|a|\log_+(a)\apprle|a|^{1/2}$ for all $a\in[-1\,,1]$,
	and the above inequality yields the desired bound for the $L^k(\Omega)$-norm of $\cT_1$.
	
	For the corresponding estimate for $\cT_2$, use the BDG inequality as follows:
	\[
		\|\cT_2\|_k^2 \le 4k\int_{\cB(t)}\left| G_{t-s}(x\,,y)-G_{t-s}(x',y)\right|^2
		\left( c(\sigma) + \Lip(\sigma)\|u(s\,,y)\|_k\right)^2\,\d s\,\d y;
	\]
	see the proof of Proposition \ref{I:pr:Lk} for more details on the
	justification of this sort of inequality. Now, Lemma \ref{lem:A3} below
	tells us that
	\[
		\sup_{t> 0}
		\int_{\cB(t)}\left| G_s(x\,,y)-G_s(x',y)\right|^2\,\d s\,\d y
		\apprle |x-x'|.
	\]
	Also,
	\begin{align*}
		&\int_{\cB(t)}\left| G_{t-s}(x\,,y)-G_{t-s}(x',y)\right|^2
			\|u(s\,,y)\|_k^2\,\d s\,\d y\\
		&\hskip2.5in\le \e^{2\beta t}\left[\cN_{\beta,k}(u)\right]^2\sup_{t>0}
			\int_{\cB(t)}\left| G_s(x\,,y)-G_s(x',y)\right|^2\,\d s\,\d y\\
		&\hskip2.5in\apprle \e^{2\beta t}\left[\cN_{\beta,k}(u)\right]^2\cdot|x-x'|.
	\end{align*}
	Therefore, it follows from the preceding development that
	\begin{equation}\label{II:T2}\begin{split}
		\|\cT_2\|_k &\apprle \sqrt{k}\left[c(\sigma) + \Lip(\sigma)
			\e^{\beta t}\cN_{\beta,k}(u)\right]\cdot |x-x'|^{1/2}.
	\end{split}\end{equation}
	This is the desired estimate of $\cT_2$.
	
	We can now combine \eqref{ITT}, \eqref{II:T1} and \eqref{II:T2}
	in order to see that
	\begin{align*}
		\frac{\|I(t\,,x) - I(t\,,x') \|_k}{|x-x'|^{1/2}} &\le
			\left[c(b)+\sqrt{k}\, c(\sigma)\right]
			+\e^{\beta t}\left[\Lip(b)+
			\sqrt{k}\, \Lip(\sigma)\right]\cN_{\beta,k}(u) \\
		&\le \sqrt{k}\,\cM_1 +  \sqrt{k}\,\e^{\beta t}\cM_2\cN_{\beta,k}(u).
	\end{align*}
	Together with \eqref{u-u:G-G:etc} and \eqref{II:G-G},
	\begin{align}\nonumber
		\frac{\|u(t\,,x) - u(t\,,x') \|_k}{|x-x'|^{\alpha\wedge (1/2)}}
		&\apprle \|u_0\|_{\C^\alpha_0}
			+ \left[c(b)+\sqrt{k}\, c(\sigma)\right]
			+\e^{\beta t}\left[\Lip(b)+
			\sqrt{k}\,\Lip(\sigma)\right]\cN_{\beta,k}(u)\\
		&\le \|u_0\|_{\C^\alpha_0} +
			\sqrt{k}\,\cM_1 +  \sqrt{k}\,\e^{\beta t}\cM_2\cN_{\beta,k}(u),
			\label{rd_e3.17}
	\end{align}
	uniformly for all $t\ge 0$, $\beta>0$, distinct $x,x'\in[0\,,1]$, and $k\ge 2$.
	We apply the preceding with the particular choice,
	\[
		\beta := 16c^4\Lip(b),
	\]
	where $c\in(0\,,\infty)$ is the same universal constant that arose
	in \eqref{I:N}. Proposition \ref{I:pr:Lk} (see in particular the
	equivalent formulation \eqref{I:N}) now tells us that
	\[
		\frac{\|I(t\,,x) - I(t\,,x') \|_k}{|x-x'|^{1/2}}
		\apprle \sqrt{k}\, \cM_1 +\sqrt{k} \,\e^{16c^4\Lip(b) t}\cM_2\cM_3,
	\]
	and
	\[
		\frac{\|u(t\,,x) - u(t\,,x') \|_k}{|x-x'|^{\alpha\wedge (1/2)}}
		\apprle\|u_0\|_{\C^\alpha_0}
		+\sqrt{k}\, \cM_1 +\sqrt{k} \,\e^{16c^4\Lip(b) t}\cM_2\cM_3,
	\]
	uniformly for all $t\ge 0$, distinct $x,x'\in[0\,,1]$,
	and $k\in[2\,, \sqrt{\Lip(b)}/\Lip(\sigma)^2]$. This is equivalent to
	the statement of the proposition.
\end{proof}

Proposition \ref{I:pr:u(x)-u(x')} has a counterpart when $k>\sqrt{\Lip(b)}/\Lip(\sigma)^2$.
We will need only the following crude version of such a counterpart.

\begin{proposition}\label{I:pr:u(x)-u(x'):bis}
	If $u_0\in\C^\alpha_0$
	for some $\alpha\in(0\,,1]$, then
	\[
		\sup_{t\geq 0}\sup_{0\le x< x'\leq 1}
		\E\left( \left|\frac{u(t\,,x) - u(t\,,x')}{(x'-x)^{\alpha\wedge(1/2)}}
		\right|^k\right) <\infty
		\qquad\text{for all $k\ge2$}.
	\]
\end{proposition}

\begin{proof}
	We merely adjust the proof of Proposition \ref{I:pr:u(x)-u(x')}
	by  using in \eqref{rd_e3.17} the result
	of Proposition \ref{I:pr:Lk:bis}, instead of Proposition \ref{I:pr:Lk},
	in order to bound $\cN_{\beta,k}(u)$. More concretely, we use
	the same argument that we used to prove Proposition \ref{I:pr:u(x)-u(x')},
	but with $\beta= 16 c^4k^2[\Lip(\sigma)]^4$
	instead of $\beta=16c^4\Lip(b)$ in that proof. Then we follow
	through the remainder of the derivation, making only small arithmetic
	adjustments for the new choice of $\beta$.
\end{proof}

Next we derive an {\it a priori} smoothness estimate for the temporal
behavior of the solution to \eqref{SHE}.

\begin{proposition}\label{I:pr:u(T)-u(t)}
 Fix $T_0 >0$.	Choose and fix some $\alpha\in(0\,,1]$,
	and define $\mu:=\min(\frac14\,,\frac12 \alpha)$.
	Then there exists a finite constant $A$---independent of
	$(b\,,\sigma)$---such that
	\[
		\sup_{x\in[0,1]}	\E\left( \left|\frac{u(T\,,x) - u(t\,,x)}{(T-t)^\mu}
		\right|^k\right) \le A^k
		\left( \|u_0\|_{\C^\alpha_0}^k +
		 k^{k/2}\left[\cM_1^k + \cM_2^k\cM_3^k\e^{Ak\Lip(b)(T)}\right]\right),
	\]
	for all $u_0\in\C^\alpha_0$, $0 \leq t < T \leq T_0$, and
	$k\in[2\,, \sqrt{\Lip(b)}/\Lip(\sigma)^2]$.
	\end{proposition}

One can make a remark, similar to  Remark \ref{rem:u(x)-u(x')},
about the Gaussian nature of the large moments of the temporal increments of
$u$ in the case that $\Lip(b)\gg[\Lip(\sigma)]^4$.

\begin{proof}
	Let $T>t>0$ and $x\in[0\,,1]$
	be fixed; the case $t=0$ is similar but simpler. In a manner similar to \eqref{u-u:G-G:etc}, we have
	\begin{equation}\label{t:u-u:G-G:etc}
		\|u(T,x) - u(t\,,x) \|_k \le \left| (\cG_Tu_0)(x) - (\cG_tu_0)(x)\right|
		+ \|\cT_1\|_k + \|\cT_2\|_k + \|\cT_3\|_k + \|\cT_4\|_k,
	\end{equation}
	where
	\begin{align*}
		\cT_1 &:= \int_{(0,t)\times(0,1)}
			\left[G_{T-s}(x\,,y)  -G_{t-s}(x\,,y)\right] b(u(s\,,y))\,\d s\,\d y,\\
		\cT_2 &:= \int_{(t,T)\times[0,1]}G_{T-s}(x\,,y)  b(u(s\,,y))\,\d s\,\d y,\\
		\cT_3 &:=\int_{(0,t)\times(0,1)}
			\left[G_{T-s}(x\,,y)  -G_{t-s}(x\,,y)\right] \sigma(u(s\,,y))\,\xi(\d s\,\d y),
			\quad\text{and}\\
		\cT_4 &:= \int_{(t,T)\times[0,1]} G_{T-s}(x\,,y)\sigma(u(s\,,y))\,
			\xi(\d s\,\d y).
	\end{align*}
	By Lemma \ref{lem:Feller:2}, $\cG_T u_0 \in \C^\alpha_0$ if
	$u_0 \in \C^\alpha_0$, and $\Vert \cG_T u_0 \Vert_{\C^\alpha_0}
	\apprle \Vert f \Vert_{\C^\alpha_0}$.
	Lemmas \ref{lem:Feller} and \ref{lem:Feller:2} ensure that
	\begin{equation}\begin{split}
		\sup_{x\in[0,1]}
			\left| (\cG_Tu_0)(x) - (\cG_tu_0)(x)\right|
			&\le \sup_{x\in[0,1]} \left| \cG_{T-t}(\cG_t u_0)(x) - (\cG_tu_0)(x)\right| \\
		&\apprle \|\cG_t u_0\|_{\C^\alpha_0}
			\cdot (T-t)^{\alpha/2} \\
		&\apprle \| u_0\|_{\C^\alpha_0} \cdot (T-t)^{\alpha/2},
	\label{II:T0}
	\end{split}\end{equation}
	uniformly for all $u_0\in\C^\alpha_0$ and $0\le t<T$.
	Next, we estimate the $L^k(\Omega)$-norms of $\cT_1,\ldots,\cT_4$, in
	this order.
	
	Lemma \ref{lem:A5} and inequality \eqref{Lip(f)} together imply that
	\begin{align*}
		\|\cT_1\|_k &\le \int_{(0,t)\times(0,1)}|G_{T-s}(x\,,y)-G_{t-s}(x\,,y)|
			\left(c(b)+\Lip(b)\|u(s\,,y)\|_k\right)\d s\,\d y\\
		&\apprle c(b)(T-t)^{1/2} + \Lip(b)\e^{\beta t}\cN_{\beta,k}(u)
			\int_{(0,t)\times(0,1)}|G_{T-s}(x\,,y)-G_{t-s}(x\,,y)|\,\d s\,\d y\\
		&\apprle\left[c(b)+\Lip(b)\e^{\beta t}\cN_{\beta,k}(u)\right]
			\cdot(T-t)^{1/2},
	\end{align*}
	for all $\beta>0$. We select $\beta:=16c^4\Lip(b)$ for the same
	constant $c$ as was used in \eqref{I:N} to deduce from \eqref{I:N} that
	\begin{equation}\label{III:T1}
		\|\cT_1\|_k \le \left[c(b)+ \cM_3\Lip(b)\e^{16c^4\Lip(b)t}\right]\cdot
		(T-t)^{1/2},
	\end{equation}
	uniformly for all $x\in[0\,,1]$, $0\le t<T$, and $k\in[2\,, \sqrt{\Lip(b)}/\Lip(\sigma)^2]$.
	
	Next we bound the size of $\cT_2$. In accord with \eqref{Lip(f)} and \eqref{int:G},
	\begin{align*}
		\|\cT_2\|_k &\le \int_{(t,T)\times[0,1]} G_{T-s}(x\,,y)
			\left(c(b)+\Lip(b)\|u(s\,,y)\|_k\right)\d s\,\d y\\
		&\le \left[c(b) + \e^{\beta T}\Lip(b)\cN_{\beta,k}(u)\right]\cdot(T-t),
	\end{align*}
	for every $\beta>0$. Once again, we choose $\beta:=16c^4\Lip(b)$ in order
	to see that
	\begin{equation}\label{III:T2}
		\|\cT_2\|_k \le\left[c(b) + \cM_3\Lip(b)\e^{16c^4\Lip(b) T}\right]\cdot(T-t),
	\end{equation}
	uniformly for all $x\in[0\,,1]$, $0\le t<T$, and $k\in[2\,, \sqrt{\Lip(b)}/\Lip(\sigma)^2]$.
	
	In order to estimate $\cT_3$, we appeal to \eqref{Lip(f)}, once again,
	together with a suitable formulation of the BDG inequality
	\cite[Proposition 4.4, p.\ 36]{Kh:CBMS}, and deduce that
	\begin{align*}
		\|\cT_3\|_k^2 &\le 4k\int_0^t\d s\int_0^1\d y\
			\left| G_{T-s}(x\,,y) - G_{t-s}(x\,,y) \right|^2 \| \sigma(u(s\,,y))\|_k^2\\
		&\apprle k\int_0^t\d s\int_0^1\d y\
			\left| G_{T-s}(x\,,y) - G_{t-s}(x\,,y) \right|^2
			\left(c(\sigma) + \Lip(\sigma)\|u(s\,,y)\|_k\right)^2\\
		&\apprle kc(\sigma)^2\cdot(T-t)^{1/2} + k[\Lip(\sigma)]^2\cdot
			\int_0^t\d s\int_0^1\d y\
			\left| G_{T-s}(x\,,y) - G_{t-s}(x\,,y) \right|^2 \|u(s\,,y)\|_k^2;
	\end{align*}
	see Lemma \ref{lem:A5} for the last inequality. We use, yet another time,
	the bound
	\[
		\|u(s\,,y)\|_k^2 \le \e^{2\beta t}[\cN_{\beta,k}(u)]^2
	\]
	[valid uniformly for all $0<s<t$, $y\in[0\,,1]$, $k\ge 2$, and $\beta>0$],
	in order to find that
	\begin{equation*}
		\|\cT_3\|_k \apprle k^{1/2}\left[c(\sigma) +
		\Lip(\sigma)\e^{\beta t}\cN_{\beta,k}(u)\right]\cdot(T-t)^{1/4}.
	\end{equation*}
	Set $\beta :=16c^4\Lip(b)$ in order to find, as before, that because of \eqref{I:N},
	\begin{equation}\label{III:T3}
		\|\cT_3\|_k \apprle k^{1/2}\left[c(\sigma) +
		\cM_3\Lip(\sigma)\e^{16c^4\Lip(b) t}\right]\cdot(T-t)^{1/4},
	\end{equation}
	uniformly for all $x\in[0\,,1]$, $0\le t<T$, and $k\in[2\,, \sqrt{\Lip(b)}/\Lip(\sigma)^2]$.
	
	Finally, we estimate $\cT_4$ by similar means: By the BDG inequality,
	\begin{align*}
		\|\cT_4\|_k^2 &\le 4k\int_t^T\d s\int_0^1\d y\
			|G_{T-s}(x\,,y)|^2 \|\sigma(u(s\,,y))\|_k^2\\
		&\apprle k\int_t^T\d s\int_0^1\d y\
			|G_{T-s}(x\,,y)|^2 \left( c(\sigma)+\Lip(\sigma)
			\|u(s\,,y)\|_k\right)^2.
	\end{align*}
	By Lemma \ref{lem:A6},
	\[
		\int_t^T\d s\int_0^1\d y\ |G_{T-s}(x\,,y)|^2\apprle (T-t)^{1/2},
	\]
	uniformly for all $x\in[0\,,1]$ and $0\le t<T$. Therefore,
	\begin{align*}
		\|\cT_4\|_k^2 &\apprle kc(\sigma)^2\cdot(T-t)^{1/2}
			+ k[\Lip(\sigma)]^2\int_t^T\d s\int_0^1\d y\ |G_{T-s}(x\,,y)|^2
			\|u(s\,,y)\|_k^2\\
		&\le kc(\sigma)^2\cdot(T-t)^{1/2}
			+ k[\Lip(\sigma)]^2\e^{2\beta T}[\cN_{\beta,k}(u)]^2\cdot
			\int_t^T\d s\int_0^1\d y\ |G_{T-s}(x\,,y)|^2\\
		&\le k\left[c(\sigma)^2+ [\Lip(\sigma)]^2\e^{2\beta T}[\cN_{\beta,k}(u)]^2
			\right]\cdot(T-t)^{1/2},
	\end{align*}
	uniformly for all $\beta>0$, $k\ge 2$, $0\le t<T$, and $x\in[0\,,1]$. Once again,
	we select $\beta:=16c^4\Lip(b)$ and appeal to \eqref{I:N} in order to see that
	\begin{equation}\label{III:T4}
		\|\cT_4\|_k \apprle k^{1/2}
		\left[c(\sigma)+ \cM_3\Lip(\sigma)\e^{16c^4\Lip(b) T}
		\right]\cdot(T-t)^{1/4},
	\end{equation}
	uniformly for all $x\in[0\,,1]$, $0\le t<T$, and $k\in[2\,, \sqrt{\Lip(b)}/\Lip(\sigma)^2]$.
	Now combine displays
	\eqref{III:T1}--\eqref{III:T4} with
	\eqref{II:T0} and \eqref{t:u-u:G-G:etc} in order to see that
	\[
		\|u(T,x) - u(t\,,x)\|_k \apprle \|u_0\|_{\C^\alpha_0}\cdot
		(T-t)^{\alpha/2} + k^{1/2}\left[\cM_1+ \cM_2
		\cM_3\e^{16c^4\Lip(b)T}\right]\cdot
		(T-t)^{1/4},
	\]
	uniformly for
	 all $k\in[2\,, \sqrt{\Lip(b)}/\Lip(\sigma)^2]$, $0\le t<T$, and $x\in[0\,,1]$.
	This has the desired result;
	we must restrict to $0\le t < T \le T_0$ in order to account for large values of $T-t$.
\end{proof}

Finally, we mention the following variation of Proposition \ref{I:pr:u(T)-u(t)}.
The following includes a bound for
the $k$th moment of temporal increments of the solution to
\eqref{SHE} when $k> \sqrt{\Lip(b)}/\Lip(\sigma)$.

\begin{proposition}\label{I:pr:u(T)-u(t):bis}
	Fix $T_0 >0$. Choose and fix $\alpha\in(0\,,1]$,
	and define $\mu:=\min(\frac14\,,\frac12 \alpha)$.
	If, in addition, $u_0\in\C^\alpha_0$, then
	\[
		\sup_{0 \le t < T \le T_0}\ \sup_{x\in[0,1]}
		\E\left( \left|\frac{u(T\,,x) - u(t\,,x)}{(T-t)^\mu}
		\right|^k\right) <\infty,
	\]
for every $k\ge2$.
\end{proposition}

\begin{proof}
	We simply adjust the proof of Proposition \ref{I:pr:u(T)-u(t)}
	by setting $\beta:=c^4 k^2[\Lip(\sigma)]^4/16$---instead
	of $\beta=16c^4\Lip(b)$---in \eqref{III:T1}, \eqref{III:T2},
	and \eqref{III:T3}. Finally, use \eqref{rde3.13}
	instead of \eqref{I:N}.
\end{proof}

We are ready to prove Theorem \ref{th:OR}.

\begin{proof}[Proof of Theorem \ref{th:OR}]
	Propositions \ref{I:pr:u(x)-u(x'):bis} and \ref{I:pr:u(T)-u(t):bis}
	and a standard application of the Kolmogorov continuity theorem
	for random fields \cite[p.\ 31]{Kunita} together imply that $u$ has a modification,
	which we continue to denote by $u$, that is H\"older continuous
	jointly in its two space-time parameters $t$ and $x$.
		
	We note that $u(t\,,0)=u(t\,,1)=0$ for all $t>0$, outside a single null set.
	By the continuity of $t\mapsto u(t)$---which we  justified
	in the previous paragraph---it suffices to prove that
	\begin{equation}\label{goal:0}
		\P\{u(t\,,0)=u(t\,,1)=0\}=1
		\qquad\text{for all $t>0$.}
	\end{equation}
	Since $G_r(0\,,y)=G_r(1\,,y)=0$ for all $r>0$ and $y\in[0\,,1]$,
	$(\cG_tu_0)(0)=(\cG_tu_0)(1)=0$, and \eqref{I:mild} implies \eqref{goal:0}.
	
	By Proposition \ref{I:pr:u(x)-u(x')}, for all $t \ge 0$, the
	function $x \mapsto I(t\,,x)$ belongs to $\cap _{\varepsilon >0}
	\C^{\frac12 - \varepsilon}_0$ and $\cG_t u_0 \in \C^\alpha_0$.
	If $\alpha \in (0,\nicefrac12)$, then
	$\cap _{\varepsilon >0} \C^{\frac12 - \varepsilon}_0 \subset \C^\alpha_0$.
	And whenever $\alpha \ge \nicefrac12$, we
	have $\C^\alpha_0 \subset \cap _{\varepsilon >0} \C^{\frac12 - \varepsilon}_0$. This proves Theorem \ref{th:OR}.
\end{proof}

\subsection{A uniform bound}

The main result of this section is the following maximal inequality.
It contains a
locally-uniform improvement to Proposition \ref{I:pr:Lk}.

\begin{theorem}\label{I:th:Lk}
	Let $u=\{u(t\,,x)\}_{t\ge0,x\in[0,1]}$ denote the continuous modification of $u$,
	and define $\varpi:=\max(12\,, 6/\alpha)$ and fix $T_0 >0$.
	If $u_0\in\C^\alpha_0$ for some $\alpha\in(0\,,1]$ and
	$\sqrt{\Lip(b)}> \varpi\Lip(\sigma)^2$,
	then 
	there exists a finite constant $A$---independent of $\Lip(b)$, $\Lip(\sigma)$---such that for all $T \in [0,T_0]$,
	\[
		\E\left( \sup_{t\in[0,T]}\sup_{x\in[0,1]}\left|
		u(t\,,x)\right|^k\right)
		\le A^k (1\vee T)^{k(1+\frac{\alpha}{2} \wedge \frac{1}{4})}
		\left(\|u_0\|_{\C^\alpha_0}^k + k^{k/2}
		\cM_1^k + k^{k/2} \cM_2^k \cM_3^k \, \e^{Ak\Lip(b)T}\right),
	\]
	uniformly for all $k\in\left(\varpi\,, \sqrt{\Lip(b)}/\Lip(\sigma)^2\right]$.
\end{theorem}

The proof of Theorem \ref{I:th:Lk} requires
a quantitative formulation of a celebrated inequality of
Garsia \cite{Garsia} (see also Garsia and Rodemich
\cite{GR}), developed by Dalang et al \cite[Proposition A.1]{DKN}.
First, let us recall
that a function $\Psi:\R\to\R_+$ is a \emph{strong Young
function} if it is even and convex on $\R$, and strictly increasing
on $\R_+$. Its inverse will be denoted by $\Psi^{-1}$.

\begin{lemma}[Garsia's lemma]\label{lem:Garsia}
	Let $(S\,,\varrho)$ be a metric space, $\nu$ a Radon measure
	on $S$, and $\Psi:\R\to\R_+$ a strong Young function that satisfies
	$\Psi(0)=0$ and $\lim_{|z|\to\infty}\Psi(z)=\infty$. Suppose
	$p:[0\,,\infty)\to\R_+$ is a continuous, strictly increasing function that
	satisfies $p(0)=0$, and choose a continuous function $f:S\to\R$.
	Then, for every compact set $K\subset S$ and for all real numbers $\delta>0$,
	\[
		\sup_{\substack{a,b\in K:\\\varrho(a,b)\le\delta}}
		|f(a)-f(b)| \le 10\sup_{w\in K}\int_0^{2\delta}\Psi^{-1}
		\left( \frac{\cC}{\left|\nu\left( B_\varrho(w\,,u/4)\right)
		\right|^2}\right)\d p(u),
	\]
	where $\Psi^{-1}(\infty):=\infty$,
	$B_\varrho(w\,,r):=\{z\in S:\, \varrho(z\,,w)<r\}$ for all
	$w\in S$ and $r>0$, and
	\[
		\cC := \int\nu(\d a)\int\nu(\d b)\
		\Psi\left( \frac{f(a) - f(b)}{p(\varrho(a\,,b))}\right).
	\]
\end{lemma}

\begin{proof}[Proof of Theorem \ref{I:th:Lk}]
	Throughout the proof, set
	$\eta = \alpha\wedge \frac12$ and $\mu:=\frac14\wedge \frac{\alpha}{2} = \frac{\eta}{2}$.
	Let $S$ denote the space-time continuum. That is,
	\[
		S := \R_+\times[0\,,1].
	\]
	We can define a metric $\varrho$ on $S$ as follows:
	\[
		\varrho\left( (s\,,y) ~,\, (t\,,x) \right) := |s-t|^\mu + |x-y|^\eta,
	\]
	for every $s,t\ge0$ and $x,y\in [0\,,1]$.
	Propositions \ref{I:pr:u(x)-u(x')} and \ref{I:pr:u(T)-u(t)} together imply that
	there exists a finite constant $A>0$ such that
	\begin{align}\nonumber
	\E\left( \left| u(s\,,y) - u(t\,,x)\right|^k\right) &\le
	    A^k \left(\|u_0\|_{\C^\alpha_0}^k + k^{k/2} \cM_1^k + k^{k/2} \cM_2^k \cM_3^k \, \e^{Ak\Lip(b)(s \vee t)}\right)\\
		  &\qquad\qquad \times \left[ \varrho\left((s\,,y)~\,,(t\,,x)\right)\right]^k,
			\label{I:inc:mom}
	\end{align}
	uniformly for every real number $k\in[2\,, \sqrt{\Lip(b)}/\Lip(\sigma)^2]$,
	all $x,y\in[0\,,1]$, and all $s,t\in[0,T]$. Choose
	and fix some
	\begin{equation}\label{I:delta}
		\delta\in \left( \frac{\varpi}{k}\,,1\right).
	\end{equation}
	This is possible because we assume $k> \varpi \geq 12$.
	We plan to apply Garsia's lemma (Lemma \ref{lem:Garsia}) with
	$p(x):=x^\delta$, $\Psi(x):=|x|^k$, and $\nu:=$ the standard Lebesgue measure
	on
	\[
		K:=[0\,,T]\times[0\,,1].
	\]
	The quantity $\cC$ of Lemma
	\ref{lem:Garsia} can now be evaluated as
	\[
		\cC = \int_{K\times K}
		\frac{|u(s\,,y)-u(t\,,x)|^k}{\left[
		\varrho\left((s\,,y)~,\,(t\,,x)\right)\right]^{k\delta}}\,\d s\,\d y\,\d t\,d x.
	\]
	We know, thanks to \eqref{I:inc:mom} and since $\delta < 1$, that $\E[\cC]<\infty$, and hence
	$\cC<\infty$ a.s. In fact, we can deduce from \eqref{I:inc:mom} and Lemma \ref{rd:lemkolm} below that
	\begin{equation}\label{I:E(C)}\begin{split}
		\left\{\E[\cC]\right\}^{1/k} & \le
		A \left(\|u_0\|_{\C^\alpha_0} + k^{1/2} \cM_1 + k^{1/2} \cM_2 \cM_3 \, \e^{AT\Lip(b)}\right)\\
		   & \qquad \times \left[\int_{K\times K} [\varrho((s\,,y)~\,,(t\,,x))]^{k(1-\delta)} \d s \d y \d t \d x\right]^{1/k} \\
		   & \apprle A \left(\|u_0\|_{\C^\alpha_0} + k^{1/2} \cM_1 + k^{1/2} \cM_2 \cM_3 \, \e^{AT\Lip(b)}\right) (1\vee T)^{(\eta(1-\delta) + 3/k)/2}
	%
	\end{split}\end{equation}
	uniformly for every $k\in[2\,,\sqrt{\Lip(b)}/\Lip(\sigma)^2]$.
	
	Next we note that, uniformly for all $(r\,,y)\in S$ and $0\le u\le 4$,
	\begin{align*}
		\nu\left( B_\varrho((r\,,y), u/4)\right) &=
			\nu\left\{ (t\,,x):\ |r-t|^\mu + |y-x|^\eta \le \frac{u}{4}\right\}\\
		&\asymp\nu\left\{ (t\,,x):\ |r-t|\apprle u^{1/\mu}\text{ and }
			|y-x|\apprle u^{1/\eta}\right\}\\
		&\asymp u^{(1/\mu)+(1/\eta)} \\
		&= u^{3/\eta}.
	\end{align*}
	In particular it follows that, uniformly for all $(r\,,y)\in S$ and $0\le u\le 4$,
	\[
		\Psi^{-1}\left( \frac{\cC}{\left|\nu\left( B_\varrho
		((r\,,y)\,, u/4)\right) \right|^2}\right) \asymp
		\frac{\cC^{1/k}}{u^{6/(\eta k)}}.
	\]
	As mentioned in the proof of Theorem \ref{th:OR}, a classical form of the Kolmogorov continuity theorem \cite[p.\ 31]{Kunita} and
	\eqref{I:inc:mom} together imply that $(t\,,x)\mapsto u(t\,,x)$
	has a continuous modification, which we again denote $u$. Therefore,
	we can now see from Lemma \ref{lem:Garsia} that
	there exist finite and nonrandom constants $L_1,L_2$ such that
	\begin{align*}
		\left| u(s\,,y) - u(t\,,x) \right|^k &\le L_1^k
			\cC \left[\int_0^{2\varrho[(s,y),(t,x)]}
			\frac{u^{\delta-1}}{u^{6/(\eta k)}}\,\d u\right]^k\\
		&\le L_2^k \cC\left[\varrho\left( (s\,,y)
			~,\, (t\,,x)\right)\right]^{k\delta-6/\eta}
			\qquad\text{a.s.}
	\end{align*}
	(where we have used that $\delta > 6/(k\eta) = \varpi/k$), uniformly for
	all $(s\,,y),(t\,,x)\in K$ and $k\in[\varpi\,, \sqrt{\Lip(b)}/\Lip(\sigma)^2]$ %
	(it might help to notice that
		\[
			L_2 = \frac{L_1\, 2^{\delta - 6/(\eta k)}}{%
			\delta - \frac{6}{k\eta}} \le \frac{L_1 \, 2^\delta}{\delta - \frac{3}{\eta}},
		\]
		since $k\ge 2$).
	In particular, \eqref{I:E(C)} implies that there exists a finite constant $A$ such that
	\begin{align}\nonumber
		&\E\left( \sup_{\substack{(s,y),(t,x)\in K:\\
			(s,y)\neq(t,x)}}
			\frac{| u(s\,,y) - u(t\,,x) |^k}{\left[\varrho\left( (s\,,y)
			~,\, (t\,,x)\right)\right]^{k\delta-6/\eta}} \right) \\
		&\qquad \le A^k \left(\|u_0\|_{\C^\alpha_0}^k + k^{k/2} \cM_1^k + k^{k/2} 				 \cM_2^k \cM_3^k \, \e^{Ak\Lip(b)T}\right)
		   (1\vee T)^{(\eta k(1-\delta) + 3)/2},
	\label{I:oops}
\end{align}
uniformly for all $k\in[\varpi\,, \sqrt{\Lip(b)}/\Lip(\sigma)^2]$. The triangle inequality implies that
\[
   \left\| \sup_{(t,x)\in K} |u(t\,,x)| \right\|_k \leq \left\| \sup_{(t,x)\in K} |u(t\,,x) - u(t\,,0)| \right\|_k + \left\| \sup_{t\in [0,T]} |u(t\,,0)| \right\|_k.
\]
The second term vanishes (see \eqref{goal:0}), and the first term is bounded above by
	\[
		\left\| \sup_{\substack{(t,x)\in K : \\ x \neq 0}} \frac{|u(t\,,x) - u(t\,,0)|}{(\vert x - 0\vert^\eta)^{\delta-6/(k\eta)}} \right\|_k \le
		\left\| \sup_{\substack{(s,y),(t,x)\in K:\\
		(s,y)\neq(t,x)}}
		\frac{| u(s\,,y) - u(t\,,x) |}{\left[\varrho\left( (s\,,y)
		~,\, (t\,,x)\right)\right]^{\delta-6/(k\eta)}} \right\|_k.
	\]
The theorem now follows from \eqref{I:oops} and the fact that $\eta (1-\delta) + 3/k \le \eta + 3/2$.
\end{proof}

\begin{lemma}\label{rd:lemkolm}
	Uniformly for $T>0$,
	\[
		\int_{K\times K} [\varrho((s\,,y)~\,,(t\,,x))]^{k(1-\delta)}
		\,\d s \,\d y \,\d t \,\d x \apprle (1\vee T)^{(k\eta(1-\delta) +3)/2}.
	\]
\end{lemma}

\begin{proof}
	The left-hand side is equal to
	\[
		\int_{K\times K} [\vert s - t\vert^{\eta/2} +
		\vert x-y\vert^\eta]^{k(1-\delta)} \,\d s\, \d y\, \d t \,\d x \\
		 \apprle \int_{K\times K} [\vert s - t\vert^{1/2} +
		 \vert x-y\vert]^{k\eta(1-\delta)}\, \d s\, \d y\, \d t\,\d x.
	\]
	Set $s-t = u$, $x,y = v$, and bound this by
	\[
		C \int_0^T \d u \int_0^1 \d v \, [u^{1/2} + v]^{k\eta(1-\delta)}.
	\]
	Let $u= w^2$, $\d u = 2 w \,\d w$, so this is bounded above by
	\[
		C \int_0^{\sqrt{T}} \d w \, w \int_0^1 \d v [w + v] ^{k\eta(1-\delta)}
		\le C \int_0^{\sqrt{T}} \d w  \int_0^1 \d v [w + v] ^{k\eta(1-\delta)+1}.
	\]
	Pass to polar coordinates in the variables $(w\,,v)$ to bound this by
	\[
		 C \int_0^{1 \vee \sqrt{T}} d\rho\, \rho^{k\eta(1-\delta)+2}
		 = \tilde C (1 \vee T)^{(k\eta(1-\delta)+3)/2}.
	\]
	This concludes the proof.
\end{proof}

\section{Proof of Theorem \ref{th:C:alpha}}\label{sec:Pf:Thm:C:alpha}
For all $N\ge1$, let $b_N$ be the following truncation of the
drift function:
\begin{equation}\label{bN}
	b_N(z) := \begin{cases}
		b(z)&\textnormal{if $|z|\le N$},\\
		b(N)&\textnormal{if $z>N$},\\
        b(-N)&\textnormal{if $z<-N$}.
	\end{cases}
\end{equation}
Let $\sigma_N(z)$ denote the corresponding truncation of the diffusion coefficient $\sigma$.

Consider the stochastic PDE
\begin{equation}\label{BNspde}
	\dot{u}_N(t\,,x) = \tfrac12 u_N''(t\,,x)
	+ b_N\left( u_N(t\,,x)\right) + \sigma_N\left(
	u_N(t\,,x)\right) \xi(t\,,x),
\end{equation}
subject to $u_N(0)=u_0$ and Dirichlet boundary conditions. Since $b_N$, $\sigma_N$ are globally Lipschitz,
standard theory \cite[Chapter 3]{Walsh} implies that the solution
$u_N$ exists for all time, has a continuous modification which we again denote by $u_N$, and is unique almost surely.
Consider also the stopping times
\begin{equation*}
	\tau_N := \inf\left\{ t>0:\ \sup_{x\in[0,1]}\left| u_N(t\,,x)\right| >
	N\right\},
\end{equation*}
where $\inf\varnothing:=\infty$.

The local property of the
stochastic integral \cite[Chapter 1]{nualartd} imply that a.s., for $N > \|u_0\|_{\L^\infty}$,
\begin{equation*}
	u_N(t\,,x) = u_{N+1}(t\,,x)\qquad\text{for all $t\in\left[0\,,
	\tau_N\right)$ and $x\in[0\,,1]$}.
\end{equation*}
Since $u_N$ is well defined for all time, and is a continuous function of
$(t\,,x)$, this proves that
$\tau_N \le \tau_{N+1}$ a.s.\ for all $N\ge \|u_0\|_{\L^\infty}$, and therefore
there exists a space-time stochastic process $u$ such that for all $N\ge \|u_0\|_{\L^\infty}$,
$u(t\,,x)=u_N(t\,,x)$ for all $x\in[0\,,1]$ and $t\in[0\,,\tau_N)$.

Consider the stopping time
\begin{equation*}
	\tau_\infty = \lim_{N\uparrow\infty}\tau_N.
\end{equation*}
Our aim is to show that $\tau_\infty = \infty$ a.s.

The continuity of $u_N$ implies that
$\sup_{x\in[0,1]}|u_N(\tau_N\,,x)|=N$ almost surely. Therefore,
the preceding readily implies the following.

\begin{lemma}
	If $u$ blows up at all, then it does so continuously. More precisely,
	\[
		\lim_{t\nearrow\, \tau_\infty}\
		\sup_{x\in[0,1]}|u(t\,,x)|=\infty
		\quad\text{a.s.\ on $\{\tau_\infty<\infty\}$.}
	\]
\end{lemma}

Now we prove Theorem \ref{th:C:alpha}.

\begin{proof}[Proof of Theorem \ref{th:C:alpha}]
We begin with the proof of the global existence. This is divided into three steps.
	
	\emph{Step 1.} In the first two steps, we replace $b$ in \eqref{SHE} with a function $\tilde b$ that
	has the following special form: There exist two constants
	$\vartheta_1,\vartheta_2\in\R$ such that $\vartheta_2\neq0$ and
	\begin{equation}\label{b}
		\tilde b(z) = \vartheta_1 +\vartheta_2
		|z|\log_+|z|\qquad\text{for all $z\in\R$},
	\end{equation}
	where we recall $\log_+(a):=\log(a\vee\e)$ for all $a\ge 0$. We may assume,
	without of generality, that
	\[
		\vartheta_2>0.
	\]
	Indeed, the case where $\vartheta_2<0$ is handled by making small adjustments to
	the ensuing argument.

	Define
	\[
		\tilde b_{N}(z) := \vartheta_1 + \vartheta_2 (|z|\wedge N)\log_+\left(|z|\wedge N\right),
	\]
	for all $N\ge3$. 
Then $|\tilde b_{N}(z)| \leq \vartheta_1 + \vartheta_2 |z| (\log N)$ so we can take
\begin{equation}\label{Lip(B_N)}
		\Lip(\tilde b_{N})=\vartheta_2 \log N .
\end{equation}
	In particular, for every fixed integer $N\ge 3$,
	the following stochastic PDE is well posed for all time:
	\[
		\dot{U}_{N}(t\,,x) = \tfrac12 U''_{N}(t\,,x) +
		\tilde b_{N}\left( U_{N}(t\,,x)\right) +
		\sigma_N\left( U_{N}(t\,,x)\right) \xi(t\,,x),
	\]
	valid for all $t>0$ and $x\in[0\,,1]$, subject to $U_{N}(0)\equiv u_0$ and the homogeneous Dirichlet boundary conditions.

	   We assume that $u_0 \in \C^\alpha_0$, where $\alpha \in (0,1]$. Define
	\[
		\tau^{(1)}_{N}
		:= \inf\left\{ t>0:\ \sup_{x\in[0,1]}|U_{N}(t\,,x)| > N\right\},
	\]
	where $\inf\varnothing:=\infty$.
	As a central part of this proof, we plan to  prove that
	\begin{equation}\label{goal:T}
		\tau^{(1)}_{\infty}
		:=\lim_{N\nearrow\infty}\tau^{(1)}_{N}=\infty
		\qquad\text{a.s.}
	\end{equation}
	
	In order to justify this assertion, note that since $\vert \sigma(z) \vert = o(\vert z \vert (\log \vert z \vert)^{1/4})$ by \eqref{rd1.2a}, we can choose $\Lip(\sigma_N) = o((\log N)^{1/4})$. Using \eqref{Lip(B_N)}, we see that
	\[
		\sqrt{\Lip(\tilde b_N)}/\Lip(\sigma_N)^2 \to \infty \qquad \mbox{as } N \to \infty,
	\]
	so the inequality
	\begin{equation}\label{NN}
		\sqrt{\Lip(\tilde b_N)} \ge \varpi\Lip(\sigma_N)^2
	\end{equation}
	holds for $N$ large enough, where $\varpi:=\max(12\,, 6/\alpha)$.
	For such $N$, take $k$ slightly
	larger than $\varpi$.
	We appeal to the Chebyshev inequality to see that
	for every $\varepsilon\in (0,1)$ and large $N$,
	\[
		\P\left\{ \tau^{(1)}_{N} < \varepsilon\right\} =
		\P\left\{ \sup_{t\in[0,\varepsilon]}\sup_{x\in[0,1]}
		|U_{N}(t\,,x)| >N\right\}
		\le N^{-k}\E\left( \sup_{t\in[0,\varepsilon]}\sup_{x\in[0,1]}
		|U_{N}(t\,,x)|^k\right).
	\]
	Next, we may apply \eqref{Lip(B_N)} and
	Theorem \ref{I:th:Lk} (recalling the formulas for $\cM_1$, $\cM_2$ and $\cM_3$ in Proposition \ref{I:pr:u(x)-u(x')}), in order to see that there exist constants $A$ and $B$ (that do not depend on $N$) such that
\begin{align*}
		\E\left( \sup_{t\in[0,\varepsilon]}\sup_{x\in[0,1]}
		|U_{N}(t\,,x)|^k \right) &\le A^k \Vert u_0 \Vert_{\C^\alpha_0}^k (B+\log N)^k\e^{A k\vartheta_2\varepsilon \log N}\\
		 &= A^k \Vert u_0 \Vert_{\C^\alpha_0}^k (B+\log N)^k N^{A k\vartheta_2\varepsilon}.
\end{align*}
In other words, we now have
	\begin{equation}\label{PTN}
		\P \{ \tau^{(1)}_{N} < \varepsilon \}
		\le A^k \Vert u_0 \Vert_{\C^\alpha_0}^k (B+\log N)^k N^{k(A \vartheta_2\varepsilon -1)},
	\end{equation}
uniformly for all integers $N$ that satisfy \eqref{NN} and $\varepsilon \in (0,1)$.
Provided that  $\varepsilon <A^{-1} \vartheta_2^{-1}$, the right-hand side converges to $0$ as $N \to \infty$, so
\eqref{PTN} implies that $\tau_{\infty}^{(1)} \ge \varepsilon$ with probability one.
	
	This in turn proves that
	\begin{equation}\label{1:T}
		\tau^{(1)}_{\infty} > \tau_0 := \frac12 \, \min(A^{-1} \vartheta_2^{-1},1)
		\qquad\text{a.s.}
	\end{equation}
	
	\emph{Step 2.} The main goal of this step is to establish the conclusions of Theorem \ref{th:C:alpha} in the special case where $b$ in \eqref{SHE} is replaced by $\tilde b$ from \eqref{b}, and to establish \eqref{goal:T}.
	
	Define
	\[
		W_t(\phi) := \int_{(0,t)\times(0,1)}\phi(x)\,\xi(\d s\,\d x)
		\qquad\text{for all $t>0$ and $\phi\in\L^2$},
	\]
	with $W_0(\phi):=0$. Then $W(\phi)$ is a Brownian motion
	for every $\phi\in\L^2$.
	Let $\F^0_t$ denote the $\sigma$-algebra generated by
	all random variables of the form $W_s(\phi)$, as $s$ ranges in
	$[0\,,t]$ and $\phi$ in $\L^2$. Let $\F_t$ denote the augmented,
	right-continuous extension of $\F^0_t$ to see that $\F:=\{\F_t\}_{t\ge0}$
	is a complete, right-continuous filtration in the sense of general theory
	of processes \cite{Sharpe}. A standard argument (see, for example Nualart and Pardoux
	\cite{NualartPardoux}) shows that
	every process $U_{N}:=\{U_{N}(t)\}_{t\ge0}$
	is a strong Markov process with
	respect to $\F$.
	
	Recall the nonrandom time $\tau_0\in(0\,,\infty)$ from
	\eqref{1:T}, and
	define $\tau^{(2)}_{\infty} := \lim_{N\nearrow\infty}\tau^{(2)}_{N}$,
	where
	\[
		\tau^{(2)}_{N} := \inf\left\{ t > \tau_0:\
		\sup_{x\in[0,1]}|U_{N}(t\,,x)|>N\right\},
	\]
	where $\inf\varnothing:=\infty$.
	According to Theorem  \ref{th:OR}, $U_N(\tau_0)$
	is almost surely an element of $\C^\alpha_0$ for some $\alpha\in(0\,,1]$.
	Therefore, we can condition on $\F_{\tau_0}$ and appeal to the
	asserted Markov property of $U_N$ in order to see that
	$\tau^{(2)}_{\infty}> 2\tau_0$ a.s. Now we proceed by induction
	in order to see that
	\begin{equation}\label{Tell}
		\tau^{(m)}_{\infty} > m\tau_0
		\qquad\text{a.s.\ for all integers $m\ge 1$,}
	\end{equation}
	where $\tau^{(m)}_{\infty}:=
	\lim_{N\nearrow\infty}\tau^{(m)}_{N}$, for
	\[
		\tau^{(m)}_{N} := \inf\left\{ t>(m-1)\tau_0:\
		\sup_{x\in[0,1]}|U_{N}(t\,,x)|>N\right\},
	\]
	with $\inf\varnothing:=\infty$. The preceding discussion reveals that
	$\tau^{(1)}_{\infty}\geq \tau^{(m)}_{\infty}
	> m\tau_0$ a.s.\ for all $m\ge1$.
	Therefore, it follows from \eqref{Tell} that $\tau^{(1)}_{\infty}=\infty$ a.s.
	This completes the proof of \eqref{goal:T}.
	
We now define
	\begin{equation*}
	  	U(t\,,x) = U_{N} (t\,,x), \qquad \mbox{for } t \in
	  	[0\,,\tau^{(1)}_N]\mbox{ and } x \in [0\,,1].
	\end{equation*}
	This defines $U(t\,,x)$ for $t \in \R_+$ and $x \in [0\,,1]$
	in a coherent way since, for each integer $N$, $U_N$ satisfies
	\begin{align}\nonumber
		U_N(t\,,x) &= (\cG_tu_0)(x) + \int_{(0,t)\times (0,1)} G_{t-s}(x\,,y)
			\tilde b_N(U_N(s\,,y))\, \d s\, \d y \\
		  & \hskip1.5in + \int_{(0,t)\times (0,1)} G_{t-s}(x\,,y)
		  	\sigma_N(U_N(s\,,y))\, \xi (\d s \,\d y).
	\label{rd4.12}
	\end{align}
	In particular, since for $|z|\leq N$, $\tilde b_N(z)=\tilde b_{N+1}(z)$ and $\sigma_N(z)=\sigma_{N+1}(z)$, as in the proof of Theorem 5.3 in \cite{cerrai-1}, we have
   $U_{N}(t\,,x) = U_{N+1}(t\,,x)$ for $t \leq \tau^{(1)}_N$.  Therefore, on $\{t \leq \tau^{(1)}_N \}$,
	\[
	   \tilde b_{N}(U_{N}(s\,,y)) = \tilde b(U(s\,,y)),
	   \qquad\mbox{and}\qquad \sigma_{N}(U_{N}(t\,,x)) = \sigma(U(t\,,x)),
	\]
	so on $\{t \le \tau^{(1)}_N\}$, \eqref{rd4.12} becomes
	\begin{align}\nonumber
		U(t\,,x) &= (\cG_tu_0)(x) + \int_{(0,t)\times (0,1)} G_{t-s}(x\,,y)
			\tilde b(U(s\,,y))\, \d s \,\d y \\
		& \hskip1.5in + \int_{(0,t)\times (0,1)}
			G_{t-s}(x\,,y)  \sigma(U(s\,,y))\, \xi (\d s\, \d y).
			\label{rd4.12a}
	\end{align}
	Since $N$ is arbitrary, this equation is satisfied for all $t \in \R_+$.
This establishes the conclusions of Theorem \ref{th:C:alpha} in the special case where $b$ in \eqref{SHE} is replaced by $\tilde b$ from \eqref{b}. We note that the stochastic integral in \eqref{rd4.12a} is a ``localized Walsh integal'' in the sense of Remark \ref{rd_rem1.8}.\\
	
\emph{Step 3.}
		Now we prove the theorem in the general case
		where $b$ is an arbitrary locally-Lipschitz function that satisfies
		the growth condition $|b(z)|=O(|z|\log|z|)$ as $|z|\to\infty$.
		
		We can find $\vartheta_1\in\R$ and $\vartheta_2>0$ such that
		\[
			b_-(z)\le b(z) \le b_+(z),\qquad\text{for all $z\in\R$},
		\]
		where
		\[
			b_{\pm}(z) := \vartheta_1 \pm \vartheta_2|z|\log_+|x|,
			\qquad\text{for all $z\in\R$}.
		\]
		
		Using Step 2, let $U_\pm(t\,,x)$ denote the solution to \eqref{SHE}, where
			$b$ is replaced by $b_{\pm}$. By analogy with \eqref{bN}, 
			let $b_{N,-}$ and $b_{N,+}$ be the truncations of $b_-$ and $b_+$, respectively. Then
	\[
	   b_{N,-}(z) \le b_N(z) \le b_{N,+}(z).
	\]
	Let $u_N$ be the solution to \eqref{BNspde}, $U_{N,-}$ (resp.~$U_{N,+}$) 
	be the solution to \eqref{BNspde} with $b_N$ replaced by $b_{N,-}$ (resp.~$b_{N,+}$). 
	According to the comparison theorem of \cite[Theorem 2.1]{DMP}, 
	for all $(t,x) \in \R_+\times [0,1]$,
	\begin{equation}\label{rd4.13}
		U_{N,-} (t\,,x) \le u_N(t\,,x) \le U_{N,+}(t\,,x).
	\end{equation}
	We have shown in Step 2 that
	\begin{equation}\label{rd4.14}
		\sup_{t\in [0,T]} \sup_{x \in [0,1]} \vert U_\pm (t,x)\vert < \infty,
		\qquad\text{for all $T>0$}.
	\end{equation}
	For any given $(t\,,x)$, for $N$ sufficiently large, $U_\pm(t,x) = U_{N,\pm}(t,x)$, therefore \eqref{rd4.13} implies that
	\begin{equation}\label{rd4.15}
		U_-(t\,,x) \le u_N(t\,,x) \le U_+(t\,,x).
	\end{equation}
	Recall that
	\[
		\tau_N = \inf\{t >0: \sup_{x \in [0,1]} \vert u_N(t,x)\vert > N \}.
	\]
	Then \eqref{rd4.14} and \eqref{rd4.15} imply that $\lim_{N \to \infty} \tau_N = \infty$ a.s.,  and we can define
	\begin{equation}\label{rd4.11}
		u(t\,,x) = u_N(t\,,x), \qquad \mbox{for } t \in [0\,,\tau_N]
		\mbox{ and } x \in [0\,,1].
	\end{equation}
	As above, this definition is coherent. By \eqref{rd4.15},
	\[
	   U_-(t,x) \leq u(t,x) \leq U_+(t,x), \qquad \mbox{for all } t \in \R_+ \mbox{ and } x \in [0,1].
	\]
	Since
	\begin{align*}
		u_N(t\,,x) &= (\cG_tu_0)(x) + \int_{(0,t)\times (0,1)}
			G_{t-s}(x\,,y)   b_N(u_N(s\,,y))\, \d s \d y \\
		& \hskip1.5in + \int_{(0,t)\times (0,1)} G_{t-s}(x\,,y)
			\sigma_N(u_N(s\,,y))\, \xi (\d s\, \d y),
	\end{align*}
	and on $\{\tau_N > t \}$, $b_N(u_N(s\,,y)) = b(u(s\,,y))$
	and $\sigma_N(u_N(s,y)) = \sigma(u(s,y))$, the local property
	of the stochastic integral \cite[Chapter 1]{nualartd} implies that on $\{\tau_N > t \}$,
	\begin{align*}
		u(t\,,x) &= (\cG_tu_0)(x) + \int_{(0,t)\times (0,1)} G_{t-s}(x\,,y)
			b(u(s\,,y))\, \d s \,\d y \\
		& \hskip1.5in + \int_{(0,t)\times (0,1)} G_{t-s}(x\,,y)
			\sigma(u(s\,,y))\, \xi (\d s\, \d y).
	\end{align*}
	Since $\P(\cup_{N \in \IN} \{\tau_N > t \}) = 1$, we see that this
	equality is satisfied a.s., and therefore $u$ is a random field
	solution of \eqref{SHE}. By \eqref{rd4.14} and \eqref{rd4.15},
	\[
		\sup_{t\in [0,T]} \sup_{x \in [0,1]} \vert u(t\,,x)\vert < \infty
		\qquad\text{for all $T>0$}.
	\]
This establishes the existence statement in Theorem \ref{th:C:alpha} as well as \eqref{rd0.1}. The solution $U$ is continuous by \eqref{rd4.11}, since each $U_N$ is continuous.

   Finally, we establish uniqueness of the solution to \eqref{SHE}. Let $v=\{v(t,x)\}$ be a solution of \eqref{SHE} (with initial condition $u_0$) in $C(\R_+ \times [0,1])$. We will show that $v=u$, where $u$ was constructed in Step 3 above. Define
\[
   \tau_N^{(1)}(v) = \inf\{ t\geq 0: \sup_{x\in [0,1]} \vert v(t,x)\vert >N\}.
\]
By sample path continuity of $v$, $\tau_N^{(1)}(v) > 0$ a.s., and $\lim_{N \to +\infty} \tau_N^{(1)}(v) = +\infty$ a.s. On $[0, \tau_N^{(1)}(v)]$, $v$ solves \eqref{BNspde}. Since $b_N$ and $\sigma_N$ are globally Lipschitz, the standard uniqueness statement implies that for all $t \geq 0$,
\[
   v(t \wedge \tau_N^{(1)}(v) \wedge \tau_N^{(1)}(u)) = u(t \wedge \tau_N^{(1)}(v) \wedge \tau_N^{(1)}(u)).
\] 
We let $N \to +\infty$. By \eqref{goal:T}, $\tau_N^{(1)}(u) \to + \infty$ a.s., so we deduce that $v(t) = u(t)$, for all $t \geq 0$. This completes the proof of Theorem \ref{th:C:alpha}.
\end{proof}

\appendix
\section{On the Green's function}
Let us solve $\dot{u} = \frac12 u''$ in $[0\,,1]$
subject to the initial condition $u_0:=\delta_y$ and
boundary conditions $u_t(0)=u_t(1)=0$ for all $t>0$. This endeavor yields
the formula \eqref{G} for the fundamental solution, which we
denote by $G_t(x\,,y)$.
In accord with the maximum principle, $G_t(x\,,y)\ge 0$ for all $t>0$
and $x,y\in[0\,,1]$.

Our next results are definitely well known, as well as simple. But
we include them since we will need to know about the parameter
dependencies.

\begin{lemma}\label{lem:A1}
	Uniformly for all $w\ge v>0$,
	\begin{equation*}
		\sum_{n=1}^\infty \left(w + vn^2\right)^{-1}
		\apprle \frac1{\sqrt{vw}}.
	\end{equation*}
\end{lemma}

\begin{proof}
	We can bound the preceding sum from above by $S_1+S_2$, where
	\begin{equation*}\begin{split}
		S_1 &:= \sum_{n=1}^{\lfloor \sqrt{w/v}\rfloor}
			w^{-1}\le (vw)^{-1/2},\textnormal{ and}\\
		S_2 &:= v^{-1}\sum_{n=1+\lfloor \sqrt{w/v}\rfloor}^\infty n^{-2} \le
			v^{-1}\int_{\sqrt{w/v}}^\infty z^{-2}\,\d z\le
			(vw)^{-1/2}.
	\end{split}\end{equation*}
	The lemma follows from these inequalities.
\end{proof}

\begin{lemma}\label{lem:A2}
	Uniformly for every $\beta>0$ and $\theta\in\{1\,,2\}$,
	\begin{equation*}
		\int_0^\infty \d t\int_0^1\d y\ \e^{-\beta t}|G_t(x\,,y)|^\theta \apprle
		\beta^{-1/\theta}.
	\end{equation*}
\end{lemma}
\begin{proof}
	The case $\theta=1$ follows simply
	because $\int_0^1G_t(x\,,y)\,\d y \le 1$.
	
	Also, the fact that $\int_0^1 [G_t(x\,,y)]^2\,\d y\le \sum_{n=1}^\infty\exp(-n^2\pi^2 t)$
	implies that
	\[
		\int_0^\infty \d t\int_0^1\d y \left[\e^{-\beta t}G_t(x\,,y)\right]^2
		\le \sum_{n=1}^\infty \int_0^\infty \e^{-2\beta t-n^2\pi^2 t}\d t
		\le  \frac12\sum_{n=1}^\infty\left( \beta + \frac{n^2\pi^2}{2}\right)^{-1}.
	\]
	The result follows from Lemma \ref{lem:A1}.
\end{proof}

\begin{lemma}\label{lem:A3}
	If $\theta\in\{1\,,2\}$, then
	\begin{equation*}
		\int_0^\infty \d t
		\int_0^1\d y\,\left| G_t(x\,,y) - G_t(x',y) \right|^\theta
		\apprle\Psi_\theta(|x-x'|),
	\end{equation*}
	uniformly for all $x,x'\in[0\,,1]$, where
	$\Psi_1(z):= z\log(\e\vee z^{-1})$ and $\Psi_2(z):=z$ for all $z>0$.
\end{lemma}

\begin{proof}
	First, let us consider the case that $\theta=2$ and $|x-x'|<\e^{-1}$.
	We may apply \eqref{G} to find that
	\begin{equation*}\begin{split}
		\int_0^1 \left[ G_t(x\,,y)-G_t(x',y)\right]^2\d y &=
			2\sum_{n=1}^\infty \left[ \sin(n\pi x)-\sin(n\pi x')\right]^2
			\e^{-n^2\pi^2 t}\\
		&\le 2\pi^2\sum_{n=1}^\infty \min\left(1\,,n^2 |x-x'|^2\right) \e^{-n^2\pi^2 t}.
	\end{split}\end{equation*}
	Therefore,
	\begin{equation*}
		\int_0^\infty\d t\int_0^1\d y\ \left[ G_t(x\,,y)-G_t(x',y)\right]^2
		\apprle\sum_{n=1}^\infty \min\left(\frac{1}{n^2}\,,|x-x'|^2\right).
	\end{equation*}
	Since $\min(r^{-2}\,,R)\le 2(r^2+R^{-1})^{-1}$ for every $r,R>0$, Lemma \ref{lem:A1}
	completes the proof in the case that $\theta=2$.
	
	If $\theta=1$, then we likewise have
	\begin{equation*}\begin{split}
		\int_0^1 | G_t(x\,,y) - G_t(x',y)|\,\d y &\apprle\sum_{n=1}^\infty
			\left| \sin(n\pi x)-\sin(n\pi x')\right|\e^{-n^2\pi^2 t/2}\\
		&\apprle\sum_{n=1}^\infty \min\left( 1\,,n|x'-x|\right)\e^{-n^2\pi^2 t/2}.
	\end{split}\end{equation*}
	Therefore,
	\begin{equation*}\begin{split}
		\int_0^\infty \d t\int_0^1\d y\  | G_t(x\,,y) - G_t(x',y)|
			&\apprle\sum_{n=1}^\infty \frac{\min\left( 1\,,n|x'-x|\right)}{n^2}\\
		&\apprle|x'-x|\sum_{n=1}^{\lfloor 1/|x'-x|\rfloor} n^{-1} +
			\sum_{n=1+\lfloor 1/|x'-x|\rfloor}^\infty n^{-2},
	\end{split}\end{equation*}
	and the result follows if $|x'-x|\le\e^{-1}$. If $|x'-x|>\e^{-1}$,
	then the very same estimates show that
	\begin{equation*}
		\int_0^\infty\d t\int_0^1\d y\
		|G_t(x\,,y) - G_t(x',y)|\apprle\sum_{n=1}^\infty n^{-2}\apprle
		|x'-x|,
	\end{equation*}
	and this completes the proof.
\end{proof}

\begin{lemma}\label{lem:A5}
	Uniformly for every $\varepsilon >0$ and $\theta\in\{1\,,2\}$,
	\begin{equation*}
		\sup_{t>0}\sup_{x\in[0,1]}
		\int_0^t\d s\int_0^1\d y\ |G_{t+\varepsilon-s}(x\,,y) - G_{t-s}(x\,,y)|^\theta
		\apprle\sqrt{\varepsilon}.
	\end{equation*}
\end{lemma}

\begin{proof}
	We first consider the case that $\theta=1$.
	We can bound $|\sin(\,\cdots)|$ from above by $1$ in \eqref{G}
	in order to find that the left-hand side is at most
	\begin{equation*}
		2\sum_{n=1}^\infty \left[
		1-\e^{-n^2\pi^2 \varepsilon/2}\right]\int_0^t \e^{-n^2\pi^2 (t-s)/2}\,\d s
		\apprle\sum_{n=1}^\infty \min\left(n^{-2}\,, \varepsilon\right),
	\end{equation*}
	since $1-\e^{-\theta}\le\min(1\,,\theta)$ for every $\theta\ge0$.
	If $\varepsilon < 1$, then a direct calculation as in the proof
	of Lemma \ref{lem:A3} shows that the series is bounded by
	$\text{const}\cdot\sqrt{\varepsilon}$.
	If $\varepsilon \ge 1$, then the series
	is a constant $c$, which is $\leq c \sqrt{\varepsilon}$. The result for $\theta = 1$ follows.
	
	For $\theta = 2$, by \eqref{G}, the left-hand side is equal to
	\begin{align*}
		&4 \int_0^t  \sin^2(n\pi x) \left(\e^{-n^2 \pi^2
			(t+\varepsilon - s)/2}  - \e^{-n^2 \pi^2 (t-s)/2}\right)^2 \d s\\
		&\hskip2in \le 4 \sum_{n=1}^\infty \left(1 - \e^{-n^2 \pi^2
			\varepsilon/2}\right)^2 \int_0^t \e^{-n^2 \pi^2 (t-s)} \,\d s \\
		&\hskip2in \apprle \sum_{n=1}^\infty \min(n^{-2},\, n^2 \varepsilon^2).
	\end{align*}
	Then we proceed as we did when $\theta=1$.
\end{proof}

\begin{lemma}\label{lem:A6}
	For all $\varepsilon>0$ and $\theta\in\{1\,,2\}$,
	\begin{equation*}
		\sup_{t\ge 0}\sup_{x\in[0,1]}
		\int_t^{t+\varepsilon}\d s\int_0^1\d y\
		\left| G_{t+\varepsilon-s}(x\,,y)\right|^\theta
		\apprle \varepsilon^{1/\theta}.
	\end{equation*}
\end{lemma}

\begin{proof}
	For $\theta = 1$, we appeal to \eqref{G} and \eqref{boundG}, to see that
	$ \int_0^1 G_{t+\varepsilon -s}(x\,,y) \,\d y\le 1,$
	and this proves the desired inequality.
	
	For $\theta = 2$, we appeal to \eqref{G} using that
	\[
		\int_0^1 [G_t(x\,,y)]^2\, \d y =
		4 \sum_{n=1}^\infty \sin^2(n\pi x) \exp(-n^2 \pi^2 t),
	\]
	to see that
	\begin{align*}
		\sup_{t>0}\sup_{x\in[0,1]}\int_t^{t+\varepsilon}\d s\int_0^1\d y\
			\vert G_{t+\varepsilon-s}(x\,,y) \vert^2
			&\apprle \sum_{n=1}^\infty  \int_0^\varepsilon\e^{-n^2\pi^2 s}\,\d s\\
		&\apprle\sum_{n=1}^\infty\left(
			\frac{1-\e^{-n^2\pi^2\varepsilon/2}}{n^2} \right),
	\end{align*}
	which is at most a constant multiple of
	$\sum_{n=1}^\infty\min(n^{-2},\varepsilon )\apprle \varepsilon^{1/2}$.
	This proves the result in the case that $\theta=2$.
\end{proof}

\begin{small}

\vskip1cm
\noindent\textbf{Robert C. Dalang.}
Institut de Math\'ematiques, Ecole Polytechnique
F\'ed\'erale de Lausanne, Station 8, CH-1015 Lausanne,
Switzerland.\\
\texttt{robert.dalang@epfl.ch}\\

\noindent\textbf{Davar Khoshnevisan.}
Department of Mathematics, The
University of Utah, 155 S. 1400 E. Salt Lake City,
UT 84112-0090, USA.\\
\texttt{davar@math.utah.edu}\\

\noindent\textbf{Tusheng Zhang.}
School of Mathematics, University of Manchester,
Oxford Road, Manchester M13 9PL, England, U.K.\\
\noindent\texttt{Tusheng.Zhang@manchester.ac.uk}
\end{small}
\end{document}